\documentclass{article}%
\usepackage{amsmath}
\usepackage{amsfonts}
\usepackage{amssymb}
\usepackage{graphicx}%
\providecommand{\U}[1]{\protect\rule{.1in}{.1in}}
\newtheorem{theorem}{Theorem}[section]

\newtheorem{corollary}[theorem]{Corollary}

\newtheorem{definition}[theorem]{Definition}
\newtheorem{example}[theorem]{Example}

\newtheorem{lemma}[theorem]{Lemma}

\newtheorem{proposition}[theorem]{Proposition}
\newtheorem{remark}[theorem]{Remark}

\newenvironment{proof}[1][Proof]{\noindent\textbf{#1.} }{\ \rule{0.5em}{0.5em}}
\begin{document}

\title{Gelfand models for classical Weyl groups}
\author{Jos\'{e} O. Araujo and Tim Bratten\\Facultad de Ciencias Exactas\\Universidad Nacional del Centro de la Provincia de Buenos Aires\\Tandil, Argentina. }
\date{}
\maketitle

\begin{abstract}
In a recent preprint Kodiyalam and Verma give a particularly simple Gelfand
model for the symmetric group that is built naturally on the space of
involutions. In this manuscript we give a natural extension of Kodiyalam and
Verma's model to a Gelfand model for Weyl groups of type $B_{n}$ and
$D_{2n+1}$. Then we define an explicit isomorphism between this Gelfand model
and the polynomial model using a technique we call telescopic decomposition.

\end{abstract}

\section{Introduction}

\noindent A Gelfand model for a finite group is a complex representation that
decomposes into a multiplicity-free sum of all the irreducible complex
representations. The terminology was introduced in \cite{soto} and alludes to
seminal work by Bernstein, Gelfand and Gelfand \cite{bgg} where models for
connected compact Lie group were constructed using induced representations.
Continuing along these lines, Klyachko \cite{klyachko} constructed models for
general linear groups over finite fields using sums of induced
representations. A recent body of work constructs natural Gelfand models for
other kinds of groups, focusing in particular on the case of finite reflection groups.

Two types of Gelfand models have emerged in the literature. The first type is
\emph{an involution model}, inspired by Klyachko's work and studied, for
example, in \cite{bad}, \cite{caselli} and \cite{inglis}. The models for the
symmetric group in \cite{adin} and the generalized symmetric group in
\cite{marberg} are versions of this kind of model. In \cite{adin2}, a
combinatorial Gelfand model is constructed for both the symmetric group and
its Iwahori-Hecke algebra. A general result about the existence of involution
models for finite Coxeter groups is treated in \cite{vinroot}. For a finite
group it is known that the dimension of a Gelfand model is equal to the number
of involutions if and only if the irreducible representations can be realized
over the real numbers. Hence criteria for the existence of generalized
involution models are studied in \cite{marberg} to treat the case of complex
reflection groups. The existence of generalized involution models for wreath
products is studied in \cite{marberg2}.

A second type of model, \emph{the polynomial model}, was introduced in
\cite{aguado1} and used to construct a Gelfand model for the symmetric group.
This second type of model is associated to a finite subgroup of the complex
general linear group, and is shown to be a Gelfand model for reflection groups
of type $B_{n}$, $D_{2n+1}$, $I_{2}\left(  n\right)  $ and $G\left(
m,1,n\right)  $ in \cite{araujo1}, \cite{araujo2} and \cite{araujo3}. Garge
and Oesterl\'{e} \cite{garge} study the polynomial model in a more general
context and give a criteria for when it is a Gelfand model for a finite
Coxeter group.

In a recent preprint \cite{kod} Kodiyalam and Verma give a particularly simple
Gelfand model for the symmetric group that is built naturally on the space of
involutions. They raise the prospect of extending their model to other Weyl
groups and of finding an explicit relationship to the polynomial model. In the
second section of this manuscript we will give natural extensions of Kodiyalam
and Verma's model to representations for the Weyl groups of type $B_{n}$ and
type $D_{n}$. We prove these extensions are Gelfand models for $W\left(
B_{n}\right)  $ and for $W(D_{2n+1})$. In the third section, we will give an
explicit isomorphism between the Gelfand models constructed in the second
section and the polynomial model for these groups by using a technique we call
telescopic decomposition.

\section{A Gelfand model}

In this section we construct natural extensions of Kodiyalam and Verma's
Gelfand model for a Weyl group of type $A_{n-1}$ to representations for the
Weyl groups of type $B_{n}$ and $D_{n}$. We prove the representation for
$W(B_{n})$ is a Gelfand model and that the representation for $W(D_{n})$ is a
Gelfand model when $n$ is odd. In what follows, $W$ will denote a Weyl group
of type $A_{n-1}$, $B_{n}$ or $D_{n}$. realized in the following manner. Let
$\mathfrak{S}_{n}$ denote the permutation group for the set of indices
$\mathbb{I}_{n}=\left\{  1,2,\ldots,n\right\}  $. We introduce the group
$W\left(  B_{n}\right)  $ as the semidirect product
\[
W\left(  B_{n}\right)  =\mathcal{C}_{2}^{n}\rtimes\mathfrak{S}_{n}%
\]
where we think of $\mathcal{C}_{2}=\left\{  \pm1\right\}  $ as subgroup of
$\mathbb{C}^{\ast}.$ Through the manuscript, when convenient, we will identify
$\mathfrak{S}_{n}$ with the subgroup $W\left(  A_{n-1}\right)  =\left\{
\boldsymbol{I}\right\}  \times\mathfrak{S}_{n}$ of $W\left(  B_{n}\right)  ,$
where $\boldsymbol{I}\in\mathcal{C}_{2}^{n}$ denotes the identity (we will use
$\boldsymbol{I}$ to denote the identity in various contexts). The group
$W\left(  D_{n}\right)  $ consists of the elements $\tau=\left(  \zeta
,\pi\right)  =\left(  \left(  \zeta_{1},\ldots,\zeta_{n}\right)  ,\pi\right)
\in\mathcal{C}_{2}^{n}\rtimes\mathfrak{S}_{n}=W\left(  B_{n}\right)  $ that
satisfy
\[
\zeta_{1}\zeta_{2}\cdots\zeta_{n}=1.
\]
There are natural inclusions
\[
\mathfrak{S}_{n}\cong W\left(  A_{n-1}\right)  \subseteq W\left(
D_{n}\right)  \subseteq W\left(  B_{n}\right)  .
\]

An element $\tau=\left(  \zeta,\pi\right)  =\left(  \left(  \zeta_{1}%
,\ldots,\zeta_{n}\right)  ,\pi\right)  \in W$ determines a partition of
$\mathbb{I}_{n}$ into the \textquotedblleft linear\textquotedblright\ and
\textquotedblleft quadratic\textquotedblright\ indices:
\[%
\begin{array}
[c]{c}%
L_{1}^{\tau}=\left\{  i\in\mathbb{I}_{n}:\pi(i)=i\text{ and }\zeta
_{i}=-1\right\} \\
L_{2}^{\tau}=\left\{  i\in\mathbb{I}_{n}:\pi(i)\neq i\text{ and }\zeta
_{\pi(i)}=-1\right\} \\
Q_{1}^{\tau}=\left\{  i\in\mathbb{I}_{n}:\pi(i)=i\text{ and }\zeta
_{i}=1\right\} \\
Q_{2}^{\tau}=\left\{  i\in\mathbb{I}_{n}:\pi(i)\neq i\text{ and }\zeta
_{\pi(i)}=1\right\}  .
\end{array}
\]

\begin{remark}
If $\tau=\left(  \zeta,\pi\right)  $ is an involution in $W$, then $\pi$ is an
involution in $\mathfrak{S}_{n}$, so the cardinalities $\left\vert L_{2}%
^{\tau}\right\vert $ and $\left\vert Q_{2}^{\tau}\right\vert $ are both even
numbers, because the elements in $L_{2}^{\tau}$ and $Q_{2}^{\tau}$ are paired
by $\pi$.
\end{remark}

The following simple proposition will be useful. Given $\tau,\mu\in W$ we
define an equivalence relation $\tau\sim_{\mathfrak{S}_{n}}\mu$ when $\tau$
and $\mu$ are conjugate under the action of $\mathfrak{S}_{n}$.

\begin{proposition}
For $\tau,\mu\in W$, we have the relation $\tau\sim_{\mathfrak{S}_{n}}\mu$ if,
and only, if $\left\vert L_{j}^{\tau}\right\vert =\left\vert L_{j}^{\mu
}\right\vert $ and $\left\vert Q_{j}^{\tau}\right\vert =\left\vert Q_{j}^{\mu
}\right\vert $ for each $j=1,2$.
\end{proposition}

\begin{proof}
Suppose $\pi\in\mathfrak{S}_{n}$ and that $\tau=\pi\mu\pi^{-1}$. Then it
follows that $L_{j}^{\tau}=\pi\left(  L_{j}^{\mu}\right)  $ and $Q_{j}^{\tau
}=\pi\left(  Q_{j}^{\mu}\right)  $ for $j=1,2$. Conversely, if the
cardinalities of the given sets are the same, then it is clear there is a
permutation $\pi\in\mathfrak{S}_{n}$ such that $\tau=\pi\mu\pi^{-1}$.
\end{proof}

\bigskip

We introduce the polynomial algebra $\mathcal{P}=\mathbb{C}\left[
x_{1},\ldots,x_{n}\right]  $ and decompose it into homogeneous components
\[
\mathcal{P=\oplus}_{k\geq0}\mathcal{P}_{k}\text{.}%
\]
A linear action of $W$ on the homogeneous component $\mathcal{P}_{1}$ of
$\mathcal{P}$ is defined by
\[
\left(  \zeta,\pi\right)  \cdot x_{i}=\zeta_{\pi\left(  i\right)  }%
x_{\pi\left(  i\right)  }%
\]
for $i\in\mathbb{I}_{n}$. The action on $\mathcal{P}_{1}$ extends canonically
to a representation of $W$ on $\mathcal{P}$ by automorphisms. There is a
corresponding representation of $W$ on the vector space
\[
V=\mathcal{P}\oplus\left(  \mathcal{P}\wedge\mathcal{P}\right)
\]
and on the symmetric algebra $S(V)$. In what follows we construct a
representation of $W$ on an invariant subspace $\mathcal{M}\subseteq$ $S(V)$.

Let $\mathcal{I}$ denote the set of involutions in $W$. To each $\tau=\left(
\zeta,\pi\right)  \in\mathcal{I}$ we associate the element $e_{\tau}\in S(V)$
defined by%
\[
e_{\tau}=\left(
%TCIMACRO{\dprod \limits_{k\in L_{1}^{\tau}}}%
%BeginExpansion
{\displaystyle\prod\limits_{k\in L_{1}^{\tau}}}
%EndExpansion
x_{k}\right)  \left(
%TCIMACRO{\dprod \limits_{\substack{m\in L_{2}^{\tau}\\m<\pi\left(  m\right)
%}}}%
%BeginExpansion
{\displaystyle\prod\limits_{\substack{m\in L_{2}^{\tau}\\m<\pi\left(
m\right)  }}}
%EndExpansion
\left(  x_{m}\wedge x_{\pi\left(  m\right)  }\right)  \right)  \left(
%TCIMACRO{\dprod \limits_{\substack{l\in Q_{2}^{\tau}\\l<\pi\left(  l\right)
%}}}%
%BeginExpansion
{\displaystyle\prod\limits_{\substack{l\in Q_{2}^{\tau}\\l<\pi\left(
l\right)  }}}
%EndExpansion
\left(  x_{l}^{2}\wedge x_{\pi\left(  l\right)  }^{2}\right)  \right)
\text{.}%
\]

The vectors $e_{\tau}$ transform nicely under the action of elements from $W$.
In particular, if $\omega=\left(  \varepsilon,\eta\right)  \in W$ then
\[
\omega e_{\tau}=\pm\left(
%TCIMACRO{\dprod \limits_{k\in L_{1}^{\tau}}}%
%BeginExpansion
{\displaystyle\prod\limits_{k\in L_{1}^{\tau}}}
%EndExpansion
x_{\eta^{-1}\left(  k\right)  }\right)  \left(
%TCIMACRO{\dprod \limits_{\substack{m\in L_{2}^{\tau}\\m<\eta\left(  m\right)
%}}}%
%BeginExpansion
{\displaystyle\prod\limits_{\substack{m\in L_{2}^{\tau}\\m<\eta\left(
m\right)  }}}
%EndExpansion
\left(  x_{\eta^{-1}\left(  m\right)  }\wedge x_{\eta^{-1}\pi\left(  m\right)
}\right)  \right)  \left(
%TCIMACRO{\dprod \limits_{\substack{l\in Q_{2}^{\tau}\\l<\eta\left(  l\right)
%}}}%
%BeginExpansion
{\displaystyle\prod\limits_{\substack{l\in Q_{2}^{\tau}\\l<\eta\left(
l\right)  }}}
%EndExpansion
\left(  x_{\eta^{-1}\left(  l\right)  }^{2}\wedge x_{\eta^{-1}\pi\left(
l\right)  }^{2}\right)  \right)
\]%
\[
=\pm e_{\eta\tau\eta^{-1}}.
\]
Let $\mathcal{M}$ be the subspace of $S(V)$ generated by the elements
$e_{\tau}$ with $\tau\in\mathcal{I}$. Thus $\mathcal{M}$ is a $W$-invariant
subspace of $S(V)$. For $\tau\in\mathcal{I}$, we let $\mathcal{M}_{\tau
}\subseteq\mathcal{M}$ be the $W$-submodule of $\mathcal{M}$ generated by the
$W$-orbit of $\tau$.

\begin{proposition}
Let $\mathcal{M}$ be the previously defined $W$-module

i) dim$_{\mathbb{C}}\left(  \mathcal{M}\right)  =\left\vert \mathcal{I}%
\right\vert $.

ii) $\mathcal{M}_{\tau}$ is the span of the vectors $\left\{  \pi e_{\tau}%
:\pi\in\mathfrak{S}_{n}\right\}  $.

iii) $\mathcal{M}_{\tau}=\mathcal{M}_{\mu}$ if and only if $\tau$ and $\mu$
are conjugate under the action of $\mathfrak{S}_{n}$.

iv) If $\mathfrak{R}$ is a system of representatives of the $\mathfrak{S}_{n}%
$-orbits in $\mathcal{I\,}$, then $\mathcal{M}=\oplus_{\rho\in\mathfrak{R}%
}\mathcal{M}_{\rho}$.\medskip
\end{proposition}

\begin{proof}
\emph{i) }For $i<j$ the factors $x_{i}$, $x_{i}\wedge x_{j}$ and $x_{i}%
^{2}\wedge x_{j}^{2}$ are linearly independent in $V$, thus it follows that
the elements $e_{\tau}$ with $\tau\in\mathcal{I}$ form a basis of
$\mathcal{M}$.

\emph{ii)} and \emph{iii) }follow from the way elements in $W$ transform the
basis vectors $e_{\tau}$ and $e_{\mu}$.

\emph{iv) }follows directly from the previous points. \emph{ }
\end{proof}

\bigskip

It is known that the dimension of a Gelfand model for a Weyl group coincides
with the number of involutions in the group. Thus it follows from the previous
proposition that $\mathcal{M}$ is a Gelfand model for $W$ if and only the
module is multiplicity-free. In order to prove this, in what follows, we will
show End$_{W}\left(  \mathcal{M},\mathcal{M}\right)  $ is a commutative ring
except for $W=W(D_{2n})$.

\begin{lemma}
\label{1}Given $\tau,\mu\in\mathcal{I}$ such that $\mathbb{I}_{n}=L_{1}^{\tau
}\cup L_{2}^{\tau}=L_{1}^{\mu}\cup L_{2}^{\mu}$ or $\mathbb{I}_{n}=Q_{1}%
^{\tau}\cup Q_{2}^{\tau}=Q_{1}^{\mu}\cup Q_{2}^{\mu}$, then there exists an
involution $\sigma\in\mathfrak{S}_{n}$ such that $\sigma e_{\tau}=\pm e_{\tau
}$ and $\sigma e_{\mu}=\mp e_{\mu}$ or else there exists an involution
$\sigma\in\mathfrak{S}_{n}$ such that $\sigma e_{\tau}=\pm e_{\mu}$.\medskip
\end{lemma}

\begin{proof}
Letting $\tau=\left(  \zeta,\pi\right)  $ and $\mu=\left(  \xi,\omega\right)
$ we will work on the involutions $\pi$ and $\omega$ in $\mathfrak{S}_{n}$.
Suppose that $\mathbb{I}_{n}=L_{1}^{\tau}\cup L_{2}^{\tau}=L_{1}^{\mu}\cup
L_{2}^{\mu}$ (this amounts to the condition that $\zeta=\xi=-\boldsymbol{I}$).
Let $G$ be the subgroup of $\mathfrak{S}_{n}$ generated by $\pi$ and $\omega$.
We consider the $G$-orbits on $\mathbb{I}_{n}$, characterizing them in the
following way. A $G$-orbit $S\subseteq\mathbb{I}_{n}$ is called \emph{trivial}
if $S\subseteq L_{1}^{\mu}\cap L_{1}^{\tau}$, $S$ is called\emph{ cyclic }if
$S\subseteq L_{2}^{\mu}\cap L_{2}^{\tau}$ and $S$ is called \emph{linear }if
$L_{2}^{\mu}\cap L_{1}^{\tau}\cap S\neq\emptyset$ or $L_{1}^{\mu}\cap
L_{2}^{\tau}\cap S\neq\emptyset$. We define \emph{a polar orbit }to be a
linear orbit with an even number of elements.

Suppose there exists a nonempty polar orbit $S$ and suppose $i\in L_{2}^{\mu
}\cap L_{1}^{\tau}\cap S$. Starting with $i$, we build a sequence by applying
first $\omega$ and then $\pi$ alternately. Let $k$ be the greatest natural
number such that this sequence is injective. We have%

\[
i=i_{1}\overset{\omega}{\rightarrow}i_{2}\overset{\pi}{\rightarrow}%
i_{3}\overset{\omega}{\rightarrow}i_{4}\overset{\pi}{\rightarrow}%
\cdots\overset{\omega}{\rightarrow}i_{k}\text{.}%
\]
The value of $\pi\left(  i_{k}\right)  $ must match one of the elements
$i_{1},i_{2},\ldots,i_{k}$. From the above table it follows that the product
of transpositions
\[
\left(  i_{2},i_{3}\right)  \left(  i_{4},i_{5}\right)  \cdots\left(
i_{k-2},i_{k-1}\right)
\]
is part of the cyclic decomposition of $\pi$. Therefore $\pi\left(
i_{k}\right)  =i_{k}$. We define involutions $\pi_{1}$, $\omega_{1}$ and
$\sigma$ in the following way:%
\[%
\begin{array}
[c]{c}%
\pi_{1}=\left(  i_{1}\right)  \left(  i_{2},i_{3}\right)  \left(  i_{4}%
,i_{5}\right)  \cdots\left(  i_{k-2},i_{k-1}\right)  \left(  i_{k}\right) \\
\omega_{1}=\left(  i_{1},i_{2}\right)  \left(  i_{3},i_{4}\right)
\cdots\left(  i_{k-1},i_{k}\right) \\
\sigma=\left(  i_{1},i_{k}\right)  \left(  i_{2},i_{k-1}\right)  \cdots\left(
i_{l},i_{l+1}\right)  \text{. }%
\end{array}
\]
Thus $\pi_{1}$ and $\omega_{1}$ are part of the cyclic decompositions of $\pi$
and $\omega$ and these involutions verify the conditions $\sigma\pi_{1}%
\sigma^{-1}=\pi_{1}$, $\sigma\omega_{1}\sigma^{-1}=\omega_{1}$. We have
\[
\sigma\left(  x_{i_{1}}\left(  x_{i_{2}}\wedge x_{i_{3}}\right)  \left(
x_{i_{4}}\wedge x_{i_{5}}\right)  \cdots\left(  x_{i_{k-2}}\wedge x_{i_{k-1}%
}\right)  x_{i_{k}}\right)  =
\]%
\[
\left(  -1\right)  ^{l-1}x_{i_{1}}\left(  x_{i_{2}}\wedge x_{i_{3}}\right)
\left(  x_{i_{4}}\wedge x_{i_{5}}\right)  \cdots\left(  x_{i_{k-2}}\wedge
x_{i_{k-1}}\right)  x_{i_{k}}\text{ \ \ and}%
\]%
\[
\sigma\left(  \left(  x_{i_{1}}\wedge x_{i_{2}}\right)  \left(  x_{i_{3}%
}\wedge x_{i_{4}}\right)  \cdots\left(  x_{i_{k-1}}\wedge x_{i_{k}}\right)
\right)  =\left(  -1\right)  ^{l}\left(  x_{i_{1}}\wedge x_{i_{2}}\right)
\left(  x_{i_{3}}\wedge x_{i_{4}}\right)  \cdots\left(  x_{i_{k-1}}\wedge
x_{i_{k}}\right)  \text{.}%
\]
Therefore if we define $\sigma$ to be the identity on the compliment of $S$ in
$\mathbb{I}_{n}$ it follows that $\sigma e_{\tau}=\pm e_{\tau}$ and $\sigma
e_{\omega}=\mp e_{\omega}$.

Suppose $G$ has no polar orbits. Let $S$ be a linear orbit with an odd number
of elements and suppose $i\in L_{2}^{\mu}\cap L_{1}^{\tau}\cap S$. As before
let $k$ be the greatest natural number such that this sequence
\[
i=i_{1}\overset{\omega}{\rightarrow}i_{2}\overset{\pi}{\rightarrow}%
i_{3}\overset{\omega}{\rightarrow}i_{4}\overset{\pi}{\rightarrow}%
\cdots\overset{\pi}{\rightarrow}i_{k}%
\]
is injective. Arguing as before, we see that $\pi(i_{k})=i_{k}$. Let $\pi_{1}%
$, $\omega_{1}$ and $\sigma$ be the involutions given by%
\[%
\begin{array}
[c]{c}%
\pi_{1}=\left(  i_{1}\right)  \left(  i_{2},i_{3}\right)  \left(  i_{4}%
,i_{5}\right)  \cdots\left(  i_{k-1},i_{k}\right) \\
\omega_{1}=\left(  i_{1},i_{2}\right)  \left(  i_{3},i_{4}\right)
\cdots\left(  i_{k-2},i_{k-1}\right)  \left(  i_{k}\right) \\
\sigma=\left(  i_{1},i_{k}\right)  \left(  i_{2},i_{k-1}\right)  \cdots\left(
i_{l},i_{l+2}\right)  \text{.}%
\end{array}
\]
Thus $\pi_{1}$ and $\omega_{1}$ are part of the cyclic decompositions of $\pi$
and $\omega$ respectively and satisfy the condition $\sigma\pi_{1}\sigma
^{-1}=\omega_{1}.$ Hence
\[
\sigma\left(  x_{i_{1}}\left(  x_{i_{2}}\wedge x_{i_{3}}\right)  \left(
x_{i_{4}}\wedge x_{i_{5}}\right)  \cdots\left(  x_{i_{k-1}}\wedge x_{i_{k}%
}\right)  \right)  =\left(  -1\right)  ^{l}\left(  x_{i_{1}}\wedge x_{i_{2}%
}\right)  \left(  x_{i_{3}}\wedge x_{i_{4}}\right)  \cdots\left(  x_{i_{k-2}%
}\wedge x_{i_{k-1}}\right)  x_{i_{k}}%
\]
so that $\sigma e_{\tau}=\pm e_{\omega}$.

Suppose now $S\subseteq L_{2}^{\mu}\cap L_{2}^{\tau}$ is a cyclic orbit. Let
$i\in S$. As before, we construct a sequence $i_{1},i_{2},\ldots,i_{k}$. \ In
this case, $k=2l$ is necessarily an even number, since from a table of type%
\[
i=i_{1}\overset{\omega}{\rightarrow}i_{2}\overset{\pi}{\rightarrow}%
i_{3}\overset{\omega}{\rightarrow}i_{4}\overset{\pi}{\rightarrow}%
\cdots\overset{\pi}{\rightarrow}i_{k}%
\]
it follows that $\omega\left(  i_{k}\right)  =i_{k}$. Since $i_{k}\in
L_{2}^{\tau}\cap L_{2}^{\mu}$ this is impossible. Therefore the table must be
of the type%
\[
i_{1}\overset{\omega}{\rightarrow}i_{2}\overset{\pi}{\rightarrow}i_{3}%
\overset{\omega}{\rightarrow}i_{4}\overset{\pi}{\rightarrow}\cdots
\overset{\omega}{\rightarrow}i_{k}%
\]
and $\pi\left(  i_{k}\right)  =i_{1}$. As in previous cases, we have the
involutions $\pi_{1}$, $\omega_{1}$ and $\sigma$ given by%
\[%
\begin{array}
[c]{c}%
\pi_{1}=\left(  i_{2},i_{3}\right)  \left(  i_{4},i_{5}\right)  \cdots\left(
i_{k-2},i_{k-1}\right)  \left(  i_{k},i_{1}\right) \\
\omega_{1}=\left(  i_{1},i_{2}\right)  \left(  i_{3},i_{4}\right)
\cdots\left(  i_{k-1},i_{k}\right) \\
\sigma=\left(  i_{2},i_{k}\right)  \left(  i_{3},i_{k-1}\right)  \cdots\left(
i_{l},i_{l+2}\right)
\end{array}
\]
which verify $\sigma\pi_{1}\sigma^{-1}=\omega_{1}$. Hence
\[
\sigma\left(  \left(  x_{i_{2}}\wedge x_{i_{3}}\right)  \left(  x_{i_{4}%
}\wedge x_{i_{5}}\right)  \cdots\left(  x_{i_{k-2}}\wedge x_{i_{k-1}}\right)
\left(  x_{i_{1}}\wedge x_{i_{k}}\right)  \right)  =
\]%
\[
\left(  -1\right)  ^{l}\left(  x_{i_{1}}\wedge x_{i_{2}}\right)  \left(
x_{i_{3}}\wedge x_{i_{4}}\right)  \cdots\left(  x_{i_{k-1}}\wedge x_{i_{k}%
}\right)  .
\]
Defining $\sigma$ to be the identity on the trivial orbits it follows that
$\sigma e_{\tau}=\pm e_{\omega\text{.}}$.

The proof for the case $\mathbb{I}_{n}=Q_{1}^{\tau}\cup Q_{2}^{\tau}%
=Q_{1}^{\mu}\cup Q_{2}^{\mu}$ is similar.
\end{proof}

\begin{remark}
When $W$ is the symmetric group $\mathfrak{S}_{n}$, this is precisely the case
$\mathbb{I}_{n}=Q_{1}^{\tau}\cup Q_{2}^{\tau}=Q_{1}^{\mu}\cup Q_{2}^{\mu}$ and
all the factors involved have the form $x_{i}^{2}\wedge x_{j}^{2}$. Moreover,
in the case where $W=W(D_{n})$ then $\left\vert L_{1}^{\tau}\right\vert $ is
an even number for every involution $\tau\in W(D_{n})$.
\end{remark}

We find it useful to introduce the symmetric bilinear form $\left\langle
\cdot,\cdot\right\rangle $ on $\mathcal{M}$ with orthonormal basis $\left\{
e_{\tau}:\tau\in\mathcal{I}\right\}  $. Using the formula for the way elements
of $W$ act on the basis vectors of $\mathcal{M}$, it follows that the form is
invariant under $W$. In particular%
\[%
\begin{array}
[c]{ccc}%
\left\langle \sigma e_{\tau},\sigma e_{\mu}\right\rangle =\pm\left\langle
e_{\sigma\tau\sigma^{-1}},e_{\sigma\mu\sigma^{-1}}\right\rangle =\delta
_{\tau,\mu} &  & \forall\tau,\mu\in\mathcal{I},\mathcal{\ \forall\sigma}\in W.
\end{array}
\]

\begin{lemma}
\label{2}If $W$ is a Weyl group of type $A_{n-1},$ $B_{n}$ or $D_{n}$ with $n$
odd in the last case, then End$_{W}\left(  \mathcal{M},\mathcal{M}\right)  $
is a commutative ring.
\end{lemma}

\begin{proof}
Fix $\phi\in$ End$_{W}\left(  \mathcal{M},\mathcal{M}\right)  $. We will show
that $\phi$ is a complex symmetric operator with respect to the invariant form
defined above$.$ In particular we will show that
\[%
\begin{array}
[c]{ccc}%
\left\langle \phi e_{\tau},e_{\mu}\right\rangle =\left\langle e_{\tau},\phi
e_{\mu}\right\rangle  &  & \forall\tau,\mu\in\mathcal{I}\text{.}%
\end{array}
\]
Let $\sigma_{i}$ be the reflection in $W\left(  B_{n}\right)  $ defined as%
\[
\sigma_{i}x_{j}=\left(  1-2\delta_{ij}\right)  x_{j}.
\]
We denote by $\eta_{i}$ the element in $W$ given by%
\[
\eta_{i}x_{j}=\left\{
\begin{array}
[c]{ccc}%
\sigma_{i}x_{j} & \text{if} & W=W\left(  B_{n}\right) \\
-\left(  \sigma_{i}x_{j}\right)  & \text{if} & W=W\left(  D_{n}\right)
\end{array}
\right.  \text{.}%
\]
This last point is where use the fact that $n$ is odd for the group $W\left(
D_{n}\right)  $, since $\eta_{i}$ does not belong to $W\left(  D_{n}\right)  $
when $n$ is even . If $\tau\in\mathcal{I}$ and $i\in\mathbb{I}_{n}$ we claim
that
\[
\eta_{i}e_{\tau}=\left\{
\begin{array}
[c]{ccc}%
e_{\tau} & \text{if} & i\notin L_{1}^{\tau}\cup L_{2}^{\tau}\\
-e_{\tau} & \text{if} & i\in L_{1}^{\tau}\cup L_{2}^{\tau}%
\end{array}
\right.  \text{.}%
\]
This is clear if $W=W\left(  B_{n}\right)  $ and in the case that $W=W\left(
D_{n}\right)  $ we have
\[
\eta_{i}e_{\tau}=\left\{
\begin{array}
[c]{ccc}%
\left(  -1\right)  ^{\left\vert L_{1}^{\tau}\cup L_{2}^{\tau}\right\vert
}e_{\tau} & \text{if} & i\notin L_{1}^{\tau}\cup L_{2}^{\tau}\\
\left(  -1\right)  ^{\left\vert L_{1}^{\tau}\cup L_{2}^{\tau}\right\vert
-1}e_{\tau} & \text{if} & i\in L_{1}^{\tau}\cup L_{2}^{\tau}%
\end{array}
\right.  .
\]
Thus the assertion follows from the fact that $\left\vert L_{1}^{\tau
}\right\vert $ and $\left\vert L_{2}^{\tau}\right\vert $ are even numbers and
the condition $L_{1}^{\tau}\cap L_{2}^{\tau}=\emptyset$.

Let $\tau,\omega\in\mathcal{I}$ and suppose $i\in\left(  L_{1}^{\tau}\cup
L_{2}^{\tau}\right)  -\left(  L_{1}^{\omega}\cup L_{2}^{\omega}\right)  $.
Then we have
\[%
\begin{array}
[c]{ccc}%
\eta_{i}e_{\tau}=-e_{\tau} & \text{and} & \eta_{i}e_{\omega}=e_{\omega
}\text{.}%
\end{array}
\]
Hence
\[
\left\langle e_{\tau},\phi e_{\omega}\right\rangle =\left\langle \eta
_{i}e_{\tau},\eta_{i}\phi e_{\omega}\right\rangle =\left\langle \eta
_{i}e_{\tau},\phi\eta_{i}e_{\omega}\right\rangle =-\left\langle e_{\tau},\phi
e_{\omega}\right\rangle
\]
that is $\left\langle e_{\tau},\phi e_{\omega}\right\rangle =0$ and similarly
$\left\langle \phi e_{\tau},e_{\omega}\right\rangle =0$. In the case that
$L_{1}^{\omega}\cup L_{2}^{\omega}=L_{1}^{\tau}\cup L_{2}^{\tau}$ we let $J$
denote this union and also write $Q_{1}^{\omega}\cup Q_{2}^{\omega}%
=Q_{1}^{\tau}\cup Q_{2}^{\tau}=K$. Let $W_{J}$ and $W_{K}$ be the subgroups of
$W$ defined by:%
\begin{align*}
W_{J}  &  =\left\{  \sigma\in W:\sigma x_{k}=x_{k},\ \forall k\in K\right\} \\
W_{K}  &  =\left\{  \sigma\in W:\sigma x_{j}=x_{j},\ \forall j\in J\right\}
\text{.}%
\end{align*}
Then $W_{J}$ and $W_{K}$ are both Weyl groups of the same type as $W$. The
involutions $\tau$ and $\omega$ can be factored as:%
\[%
\begin{array}
[c]{ccc}%
\tau=\tau_{J}\ \tau_{K} & \text{and} & \omega=\omega_{J}\ \omega_{K}%
\end{array}
\]
in $W_{J}\times W_{K}$. Applying Lemma \ref{1} to the pairs $e_{\tau_{J}}$,
$e_{\omega_{J}}$ and $e_{\tau_{K}}$, $e_{\omega_{K}}$, we conclude that there
is an involution $\omega$ such that:%
\[%
\begin{array}
[c]{ccc}%
\omega e_{\tau}=-e_{\tau}\text{ and }\omega e_{\omega}=e_{\omega} & \text{or}
& \omega e_{\tau}=\pm e_{\omega}\text{.}%
\end{array}
\]
In the first case, as before, $\left\langle e_{\tau},\phi e_{\omega
}\right\rangle =0=\left\langle e_{\omega},\phi e_{\omega}\right\rangle $ and
in the second case we have:%
\[
\left\langle e_{\tau},\phi e_{\omega}\right\rangle =\left\langle \omega
e_{\tau},\omega\phi e_{\omega}\right\rangle =\left\langle \pm e_{\omega}%
,\pm\phi e_{\tau}\right\rangle =\left\langle e_{\omega},\phi e_{\tau
}\right\rangle =\left\langle \phi e_{\tau},e_{\omega}\right\rangle \text{.}%
\]
Hence End$_{W}\left(  \mathcal{M},\mathcal{M}\right)  $ is an algebra of
operators contained in the space of symmetric operators. It follows that
End$_{W}\left(  \mathcal{M},\mathcal{M}\right)  $ is commutative.\medskip
\end{proof}

\begin{theorem}
Suppose $W$ is a Weyl group of type $A_{n},B_{n}$ or $D_{2n+1}$, then
$\mathcal{M}$ is a Gelfand Model for $W$.\medskip
\end{theorem}

\begin{proof}
By Lemma \ref{2} $\mathcal{M}$ is a multiplicity-free $W$-module. On the other
hand, the dimension of $\mathcal{M}$ coincides with the number of involutions.
As mentioned above, this is known to be the dimension of a Gelfand model for a
Weyl group.
\end{proof}

\begin{remark}
Since Baddeley has shown in his doctoral thesis that a Weyl group of type
$D_{2n}$ does not have an involution model (see, for example \cite{vinroot}),
one would expect that the representation defined in this section is not a
Gelfand model for $W(D_{2n})$. We note that our module $\mathcal{M}$ is not
mutiplicity-free in this case. To explain, introduce the following notation.
If $\tau=(\zeta,\pi)\in W(B_{n})$ define $-\tau=(-\zeta,\pi)$. Now observe
that if $\boldsymbol{I}$ denotes the identity in $W(D_{n})$ then both
$e_{\boldsymbol{I}}=1$ as well as $e_{-\boldsymbol{I}}=x_{1}\cdots x_{2n}$
generate the trivial representation for $W(D_{n})$, the point being that
$-\boldsymbol{I}$ does not belong to $W(D_{n})$ when $n$ is odd. In the
following section we will show, in general, that $\mathcal{M}_{\tau}%
\cong\mathcal{M}_{-\tau}$ for the group $W(D_{n})$, the difference being that
when $n$ is odd then $-\tau$ is not in $W(D_{n})$ for $\tau\in\mathcal{I}$.
\end{remark}

\section{\noindent Relation to the polynomial model}

Suppose $G\subseteq GL_{n}\left(  \mathbb{C}\right)  $ is a finite subgroup.
The group $G$ acts on the space of complex valued polynomials $\mathcal{P}%
=\mathbb{C}\left[  x_{1},\ldots,x_{n}\right]  $ according to the natural
action of $GL(n,\mathbb{C})$. The polynomial model for the group $G$ is based
on a subspace $\mathcal{N}_{G}\subseteq\mathcal{P}$ determined as the zeros of
an associated subalgebra of the $G$-invariant differential operators defined
on $\mathcal{P}$. In particular, let $\mathcal{D}$ denote the Weyl algebra of
differential operators with polynomial coefficients. Then $G$ acts naturally
on $\mathcal{D}$ by
\[%
\begin{array}
[c]{ccc}%
g\cdot D=gDg^{-1} &  & g\in G,D\in\mathcal{D}%
\end{array}
\]
where both $g$ and $D$ are taken as elements in End$_{\mathbb{C}}\left(
\mathcal{P}\right)  $. Let $\mathcal{D}^{G}$ denote the centralizer of $G$ in
$\mathcal{D}$. We will define a $G$ invariant subset of elements of negative
degree in $\mathcal{D}^{G}$. Let $\mathbb{I}_{n}=\left\{  1,2,\ldots
,n\right\}  $ be the set of indices and let $M=\left\{  \alpha:\mathbb{I}%
_{n}\rightarrow\mathbb{N}_{0}=\mathbb{N\cup}\left\{  0\right\}  \right\}  $
be\emph{ the set of multi-indices in} $\mathbb{I}_{n}$. Given $\alpha\in M$ we
put:
\[
\left\vert \alpha\right\vert =\sum_{i=1}^{n}\alpha(i).
\]
Each element in $\mathcal{D}$ can be written uniquely in the form
\[
\sum_{\alpha,\beta\in M}a_{\alpha,\beta}\ x^{\alpha}\partial^{\beta}.
\]
We define \emph{the algebra} $\mathcal{D}_{-}^{G}$ \emph{of invariant
operators of negative degree }as follows%
\[
\mathcal{D}_{-}^{G}=\left\{  D\in\mathcal{D}^{G}:D=\sum_{\left\vert
\beta\right\vert >\left\vert \alpha\right\vert }a_{\alpha,\beta}\ x^{\alpha
}\partial^{\beta}\right\}  \text{.}%
\]
Observe that if
\[
V_{m}=%
%TCIMACRO{\dbigoplus \limits_{k=0}^{m}}%
%BeginExpansion
{\displaystyle\bigoplus\limits_{k=0}^{m}}
%EndExpansion
\mathcal{P}_{k}%
\]
then $\mathcal{D}_{-}^{G}$ determines a subalgebra of nilpotent operators in
Hom$_{G}(V_{m})$.

The space $\mathcal{N}_{G}$ is defined by
\[
\mathcal{N}_{G}=\left\{  P\in\mathcal{P}:D\left(  P\right)  =0,\forall
D\in\mathcal{D}_{-}^{G}\right\}  .
\]
\ 

\noindent In general one knows that $\mathcal{N}_{G}$ contains a Gelfand model
for $G$ \cite{aguado2}. As mentioned in the introduction, when $G$ is a Weyl
group of type $A_{n}$, $B_{n}$, or $D_{2n+1}$ realized in linear form by its
geometric representation then $\mathcal{N}_{G}$ is a Gelfand model for $G$. We
note when $G=W$ is a Weyl group of type $A_{n-1}$, $B_{n}$, or $D_{n}$ then
the geometric representation is defined in $\mathbb{C}^{n}$ and is given by
\[
\left(  \zeta,\pi\right)  \cdot e_{j}=\zeta_{\pi(j)}e_{\pi(j)}%
\]
where $\left\{  e_{1,}\ldots,e_{n}\right\}  $ is the natural basis. In
particular, the corresponding action of $W$ in the space of polynomials
$\mathcal{P}$ is nothing but the action defined in the previous section.

In this section we will construct a specific isomorphism between the Gelfand
model defined in the previous section and the corresponding polynomial model.
We continue to let $\mathcal{I}$ denote the set of involutions in $W$.
Recalling the notation from the previous section, for every $\tau
\in\mathcal{I}$ we have a corresponding decomposition of $\mathbb{I}_{n}$ into
disjoint subsets $L_{1}^{\tau}$, $L_{2}^{\tau}$, $Q_{1}^{\tau}$ and
$Q_{2}^{\tau}$. As in the previous section we write $\tau=\left(  \zeta
,\pi\right)  $. Let
\[
\pi=\pi_{1}\pi_{2}\cdots\pi_{m}%
\]
be the cyclic decomposition of $\pi$. Since $\tau$ is an involution each cycle
$\pi_{j}$ is a transposition, say $\pi_{j}=(k,l)$. It also follows that either
$\zeta_{k}$ and $\zeta_{l}$ are both $-1$ or both $1,$ that is:
\[
\left\{  k,l\right\}  \subseteq L_{2}^{\tau}\text{ or }\left\{  k,l\right\}
\subseteq Q_{2}^{\tau}\text{.}%
\]
If $\left\{  k,l\right\}  \subseteq Q_{2}^{\tau}$ define $\tau_{j}%
=(\boldsymbol{I},\pi_{j})$. In this case we say $\tau_{j}$ is \emph{a positive
cycle.} When $\left\{  k,l\right\}  \subseteq L_{2}^{\tau}$ define $\tau
_{j}=(\xi,\pi_{j})$ where $\xi_{i}=-1$ for $i\in\left\{  k,l\right\}  $ and
$\xi_{i}=1$ if $i\notin\left\{  k,l\right\}  $ and call $\tau_{j}$ \emph{a
negative cycle}. Let $\tau^{+}\in W$ be the product of the positive cycles
that decompose $\tau$ and let $\tau^{-}\in W$ be the product of the negative
cycles that decompose $\tau$. Thus $\tau=\tau^{+}\tau^{-}$ when $L_{1}^{\tau
}=\emptyset$ and in general $\tau=\sigma\tau^{+}\tau^{-}$ where $\sigma
=\left(  \zeta,\boldsymbol{I}\right)  $ and
\[
\zeta_{i}=\left\{
\begin{array}
[c]{ccc}%
-1 & \text{if} & i\in L_{1}^{\tau}\\
1 & \text{if} & i\notin L_{1}^{\tau}%
\end{array}
\right.  \text{.}%
\]
In a key construction that follows, we will utilize the subgroup
$\vartheta_{\tau}=\vartheta_{\tau}^{+}\times\vartheta_{\tau}^{-}$ of
$\mathfrak{S}\left(  Q_{2}^{\tau}\right)  \times\mathfrak{S}\left(
L_{2}^{\tau}\right)  $ where $\vartheta_{\tau}^{+}$ and $\vartheta_{\tau}^{-}$
are the centralizers of $\tau^{+}$ and $\tau^{-}$ in $\mathfrak{S}\left(
Q_{2}^{\tau}\right)  $ and $\mathfrak{S}\left(  L_{2}^{\tau}\right)  $, respectively.

To each subgroup $\mathfrak{K\subseteq S}_{n}$ we can associate the operator
$\Omega_{\mathfrak{K}}$ defined as%
\[
\Omega_{\mathfrak{K}}=\sum_{\kappa\in\mathfrak{K}}sgn\left(  \kappa\right)
\kappa\text{.}%
\]
Observe that this operator can be applied naturally to the elements of any
$\mathfrak{S}_{n}$-module. \ 

\begin{remark}
Given $\pi\in\mathfrak{S}_{n}$ note that:%
\[
\pi\Omega_{\mathfrak{K}}\pi^{-1}=\Omega_{\pi\mathfrak{K\pi}^{-1}}%
\]
so that%
\[
\pi\Omega_{\mathfrak{K}}=\Omega_{\mathfrak{K}}\pi=sg\left(  \pi\right)
\Omega_{\mathfrak{K}}\text{ if }\pi\in\mathfrak{K.}%
\]

\end{remark}

Suppose $\tau=(\zeta,\omega)$ is an involution. Then we can order the sets
\[%
\begin{array}
[c]{ccc}%
L_{2}^{\tau}=\left\{  i_{1},j_{1},\ldots,i_{r},j_{r}\right\}  & \text{and} &
Q_{2}^{\tau}=\left\{  k_{1},l_{1},\ldots,k_{s},l_{s}\right\}
\end{array}
\]
such that the sequences $i_{1},\ldots,i_{r}$ and $k_{1},\ldots,k_{s}$ are
strictly increasing and where $i_{p}<j_{p}=\omega\left(  i_{p}\right)  $,
$k_{q}<l_{q}=\omega\left(  k_{q}\right)  $ for $1\leq p\leq r$, $1\leq q\leq
s$. We define a corresponding multi-index $\alpha_{\tau}\in M$ in the
following manner%

\[
\alpha_{\tau}\left(  i\right)  =\left\{
\begin{array}
[c]{ccc}%
1 & \text{if} & i\in L_{1}^{\tau}\\
0 & \text{if} & i\in Q_{1}^{\tau}%
\end{array}
\right.  \text{ and }%
\begin{array}
[c]{ccc}%
\alpha_{\tau}\left(  i_{p}\right)  =4p-3 &  & \alpha_{\tau}\left(
j_{p}\right)  =4p-1\\
\alpha_{\tau}\left(  k_{q}\right)  =4q-4 &  & \alpha_{\tau}\left(
l_{q}\right)  =4q-2
\end{array}
\text{.}%
\]

To an involution $\tau\in W$ we associate the polynomial%
\[
P_{\tau}=\Omega_{\vartheta_{\tau}}x^{\alpha_{\tau}}%
\]
where $\vartheta_{\tau}=\vartheta_{\tau}^{+}\times\vartheta_{\tau}^{-}$ with
$\vartheta_{\tau}^{+}$ and $\vartheta_{\tau}^{-}$ are the centralizers of
$\tau^{+}$ and $\tau^{-}$ in $\mathfrak{S}\left(  Q_{2}^{\tau}\right)  $ and
$\mathfrak{S}\left(  L_{2}^{\tau}\right)  $, respectively.

\begin{remark}
In the simple case $L_{1}^{\tau}=\emptyset$, $\tau^{-}=\left(  \left(
-1,-1,1.\cdots,1\right)  ,\left(  12\right)  \right)  $ and $\tau^{+}=\left(
\boldsymbol{I},(34)\right)  $. Then $P_{\tau}$ is exactly the product of the
two determinants
\[
P_{\tau}=\left\vert
\begin{array}
[c]{cc}%
x_{1} & x_{2}\\
x_{1}^{3} & x_{2}^{3}%
\end{array}
\right\vert \left\vert
\begin{array}
[c]{cc}%
1 & 1\\
x_{3}^{2} & x_{4}^{2}%
\end{array}
\right\vert .
\]
In general, given the sets $L_{2}^{\tau}=\left\{  i_{1},j_{1},\ldots
,i_{r},j_{r}\right\}  $ and $Q_{2}^{\tau}=\left\{  k_{1},l_{1},\ldots
,k_{s},l_{s}\right\}  $ then the product of the two determinants is given by
\[
\left\vert
\begin{array}
[c]{ccccc}%
x_{i_{1}} & x_{j_{1}} & \cdots & x_{i_{r}} & x_{j_{r}}\\
x_{i_{1}}^{3} & x_{j_{1}}^{3} & \cdots & x_{i_{r}}^{3} & x_{j_{r}}^{3}\\
x_{i_{1}}^{5} & x_{j_{1}}^{5} & x_{i_{2}}^{5} & \cdots & x_{j_{r}}^{5}\\
\vdots & \vdots & \vdots & \ddots & \vdots\\
x_{i_{1}}^{4r-1} & x_{j_{1}}^{4r-1} & x_{i_{2}}^{4r-1} & \cdots & x_{j_{r}%
}^{4r-1}%
\end{array}
\right\vert \left\vert
\begin{array}
[c]{ccccc}%
1 & 1 & \cdots & 1 & 1\\
x_{k_{1}}^{2} & x_{l_{1}}^{2} & \cdots & x_{k_{s}}^{2} & x_{l_{s}}^{2}\\
x_{k_{1}}^{4} & x_{l_{1}}^{4} & x_{k_{2}}^{4} & \cdots & x_{l_{s}}^{4}\\
\vdots & \vdots & \vdots & \ddots & \vdots\\
x_{k_{1}}^{4s-2} & x_{l_{1}}^{4s-2} & x_{k_{2}}^{4s-2} & \cdots & x_{l_{s}%
}^{4s-2}%
\end{array}
\right\vert =
\]%
\[%
%TCIMACRO{\dsum \limits_{\pi\in\mathfrak{S}\left(  L_{2}^{\tau}\right)
%\times\mathfrak{S}\left(  Q_{2}^{\tau}\right)  }}%
%BeginExpansion
{\displaystyle\sum\limits_{\pi\in\mathfrak{S}\left(  L_{2}^{\tau}\right)
\times\mathfrak{S}\left(  Q_{2}^{\tau}\right)  }}
%EndExpansion
sgn\left(  \pi\right)  \pi\cdot x_{i_{1}}x_{j_{1}}^{3}x_{i_{2}}^{5}\cdots
x_{j_{r}}^{4r-1}x_{l_{1}}^{2}x_{k_{2}}^{4}\cdots x_{l_{s}}^{4s-2}%
\]
whereas, assuming $L_{1}^{\tau}=\emptyset$,
\[
P_{\tau}=%
%TCIMACRO{\dsum \limits_{\pi\in\vartheta_{\tau}^{-}\times\vartheta_{\tau}^{+}%
%}}%
%BeginExpansion
{\displaystyle\sum\limits_{\pi\in\vartheta_{\tau}^{-}\times\vartheta_{\tau
}^{+}}}
%EndExpansion
sgn\left(  \pi\right)  \pi\cdot x_{i_{1}}x_{j_{1}}^{3}x_{i_{2}}^{5}\cdots
x_{j_{r}}^{4r-1}x_{l_{1}}^{2}x_{k_{2}}^{4}\cdots x_{l_{s}}^{4s-2}%
\]
is the signed sum over the smaller group of permutations.
\end{remark}

It will be useful to consider the following definition. Suppose we have two
disjoint ordered subsets of indices $I=\left\{  i_{1},i_{2},\ldots
,i_{r}\right\}  $, $J=\left\{  j_{1},j_{2},\ldots,j_{r}\right\}  $ of the same
cardinality and let $\pi$ be a permutation of $I$. Then we can extend $\pi$ to
a permutation $\overline{\pi}$ of $I\cup J$ by defining
\[
\overline{\pi}(j_{k})=j_{l}\text{ if }\pi(i_{k})=i_{l}\text{.}%
\]
The permutation $\overline{\pi}$ will be called the \emph{corresponding double
permutation. }

\begin{example}
Suppose $S=\left\{  1,2,3,4,5,6\right\}  $. We divide $S$ into two ordered
sets
\[
I=\left\{  i_{1},i_{2},i_{3}\right\}  =\left\{  2,4,1\right\}  \ \text{and
}J=\left\{  j_{1},j_{2},j_{3}\right\}  =\left\{  3,6,5\right\}
\]
such that $i_{k}<j_{k}$ for each $k$. Consider the involution
\[
\tau=(23)(46)(15).
\]
Let $\pi$ be the permutation of $I$ such that $\pi(i_{1})<\pi(i_{2})<\pi
(i_{3})$ and let $\overline{\pi}$ be the corresponding double permutation
of\emph{ }$S$. Thus $\overline{\pi}$ commutes with $\tau$.
\end{example}

Let $\mathfrak{M}\subseteq\mathcal{P}$ denote the $W$-submodule generated by
the polynomials $P_{\tau}$ with $\tau\in\mathcal{I}$. The following
proposition shows that the linear map determined by the assignment
\[
e_{\tau}\mapsto P_{\tau}%
\]
defines an isomorphism between $\mathfrak{M}$ and a submodule of the Gelfand
module $\mathcal{M}$ constructed in the previous section (latter on we will
see that $\mathcal{M}\cong\mathfrak{M}$).

\begin{proposition}
The linear application $\phi:\mathcal{M}\rightarrow\mathfrak{M}$ defined on
the basis vectors by the formula $\phi\left(  e_{\tau}\right)  =P_{\tau}$, for
$\tau\in\mathcal{I}$, is a morphism of $W$-modules.
\end{proposition}

\begin{proof}
Suppose $\tau\in\mathcal{I}$ and let $\sigma=\left(  \zeta,\boldsymbol{I}%
\right)  \in W$. From our definition of the action of $W$ on the basis vectors
in $\mathcal{M}$ it follows that%

\[
\sigma e_{\tau}=\left(
%TCIMACRO{\dprod _{j\in L_{1}^{\tau}\cup L_{2}^{\tau}}}%
%BeginExpansion
{\displaystyle\prod_{j\in L_{1}^{\tau}\cup L_{2}^{\tau}}}
%EndExpansion
\zeta_{j}\right)  e_{\tau}\text{.}%
\]
Decomposing the operator $\Omega_{\vartheta_{\tau}}=\Omega_{\vartheta_{\tau
}^{-}}\times\Omega_{\vartheta_{\tau}^{+}}$ and applying this to the
polynomial
\[
x^{\alpha_{\tau}}=\left(
%TCIMACRO{\dprod \limits_{j\in L_{1}^{\tau}}}%
%BeginExpansion
{\displaystyle\prod\limits_{j\in L_{1}^{\tau}}}
%EndExpansion
x_{j}\right)  \left(
%TCIMACRO{\dprod \limits_{j\in L_{2}^{\tau}}}%
%BeginExpansion
{\displaystyle\prod\limits_{j\in L_{2}^{\tau}}}
%EndExpansion
x_{j}^{\alpha_{\tau}(j)}\right)  \left(
%TCIMACRO{\dprod \limits_{j\in Q_{2}^{\tau}}}%
%BeginExpansion
{\displaystyle\prod\limits_{j\in Q_{2}^{\tau}}}
%EndExpansion
x_{j}^{\alpha_{\tau}(j)}\right)
\]
it follows that
\[
P_{\tau}=\left(
%TCIMACRO{\dprod \limits_{j\in L_{1}^{\tau}}}%
%BeginExpansion
{\displaystyle\prod\limits_{j\in L_{1}^{\tau}}}
%EndExpansion
x_{j}\right)  \Omega_{\vartheta_{\tau}^{-}}\left(
%TCIMACRO{\dprod \limits_{j\in L_{2}^{\tau}}}%
%BeginExpansion
{\displaystyle\prod\limits_{j\in L_{2}^{\tau}}}
%EndExpansion
x_{j}^{\alpha_{\tau}(j)}\right)  \Omega_{\vartheta_{\tau}^{+}}\left(
%TCIMACRO{\dprod \limits_{j\in Q_{2}^{\tau}}}%
%BeginExpansion
{\displaystyle\prod\limits_{j\in Q_{2}^{\tau}}}
%EndExpansion
x_{j}^{\alpha_{\tau}(j)}\right)  \text{.}%
\]
Thus, using that fact that $\Omega_{\vartheta_{\tau}^{-}}$ is linear, we have%

\[
\sigma P_{\tau}=\left(
%TCIMACRO{\dprod _{j\in L_{1}^{\tau}}}%
%BeginExpansion
{\displaystyle\prod_{j\in L_{1}^{\tau}}}
%EndExpansion
\zeta_{j}\right)  \left(
%TCIMACRO{\dprod \limits_{j\in L_{1}^{\tau}}}%
%BeginExpansion
{\displaystyle\prod\limits_{j\in L_{1}^{\tau}}}
%EndExpansion
x_{j}\right)  \Omega_{\vartheta_{\tau}^{-}}\left(  \left(
%TCIMACRO{\dprod _{j\in L_{2}^{\tau}}}%
%BeginExpansion
{\displaystyle\prod_{j\in L_{2}^{\tau}}}
%EndExpansion
\zeta_{j}\right)  \left(
%TCIMACRO{\dprod \limits_{j\in L_{2}^{\tau}}}%
%BeginExpansion
{\displaystyle\prod\limits_{j\in L_{2}^{\tau}}}
%EndExpansion
x_{j}^{\alpha_{\tau}(j)}\right)  \right)  \Omega_{\vartheta_{\tau}^{+}}\left(
%
%TCIMACRO{\dprod \limits_{j\in Q_{2}^{\tau}}}%
%BeginExpansion
{\displaystyle\prod\limits_{j\in Q_{2}^{\tau}}}
%EndExpansion
x_{j}^{\alpha_{\tau}(j)}\right)  =
\]%
\[
\left(
%TCIMACRO{\dprod _{j\in L_{1}^{\tau}}}%
%BeginExpansion
{\displaystyle\prod_{j\in L_{1}^{\tau}}}
%EndExpansion
\zeta_{j}\right)  \left(
%TCIMACRO{\dprod _{j\in L_{2}^{\tau}}}%
%BeginExpansion
{\displaystyle\prod_{j\in L_{2}^{\tau}}}
%EndExpansion
\zeta_{j}\right)  \left(
%TCIMACRO{\dprod \limits_{j\in L_{1}^{\tau}}}%
%BeginExpansion
{\displaystyle\prod\limits_{j\in L_{1}^{\tau}}}
%EndExpansion
x_{j}\right)  \Omega_{\vartheta_{\tau}^{-}}\left(
%TCIMACRO{\dprod \limits_{j\in L_{2}^{\tau}}}%
%BeginExpansion
{\displaystyle\prod\limits_{j\in L_{2}^{\tau}}}
%EndExpansion
x_{j}^{\alpha_{\tau}(j)}\right)  \Omega_{\vartheta_{\tau}^{+}}\left(
%TCIMACRO{\dprod \limits_{j\in Q_{2}^{\tau}}}%
%BeginExpansion
{\displaystyle\prod\limits_{j\in Q_{2}^{\tau}}}
%EndExpansion
x_{j}^{\alpha_{\tau}(j)}\right)  =\left(
%TCIMACRO{\dprod _{j\in L_{1}^{\tau}\cup L_{2}^{\tau}}}%
%BeginExpansion
{\displaystyle\prod_{j\in L_{1}^{\tau}\cup L_{2}^{\tau}}}
%EndExpansion
\zeta_{j}\right)  P_{\tau}\text{.}%
\]
Therefore $\phi(\sigma e_{\tau})=\sigma\phi\left(  e_{\tau}\right)  $.

Now consider an element in $W$ of the form $\left(  \boldsymbol{I},\pi\right)
$. Then
\[
\pi e_{\tau}=\left(  -1\right)  ^{\varepsilon_{\pi}}e_{\pi\tau\pi^{-1}}%
\]
where $\varepsilon_{\pi}$ is the number of transpositions $\left(  i,j\right)
$ which are factors of $\tau$ with $\pi\left(  i\right)  >\pi\left(  j\right)
$. Using the formula in Remark 3.1 we obtain
\[
\pi P_{\tau}=\pi\Omega_{\vartheta_{\tau}}\pi^{-1}\pi x^{\alpha_{\tau}}%
=\Omega_{\pi\vartheta_{\tau}\pi^{-1}}\pi x^{\alpha_{\tau}}\text{.}%
\]
Consider the sets
\[
L_{1}^{\tau}=\left\{  m_{1},\ldots,m_{t}\right\}  ,\text{ }L_{2}^{\tau
}=\left\{  i_{1},j_{1},\ldots,i_{r},j_{r}\right\}  \text{ and }Q_{2}^{\tau
}=\left\{  k_{1},l_{1},\ldots,k_{s},l_{s}\right\}
\]
ordered in the previously defined way, so that
\[
x^{\alpha_{\tau}}=x_{m_{1}}\cdots x_{m_{t}}x_{i_{1}}x_{j_{1}}^{3}\cdots
x_{i_{r}}^{4r-3}x_{k_{1}}^{0}x_{l_{1}}^{2}\cdots x_{l_{s}}^{4s-2}.
\]
Thus
\[
\pi x^{\alpha_{\tau}}=x_{\pi(m_{1})}\cdots x_{\pi(m_{t})}x_{\pi(i_{1})}%
x_{\pi(j_{1})}^{3}\cdots x_{\pi(i_{r})}^{4r-3}x_{\pi(k_{1})}^{0}x_{\pi(l_{1}%
)}^{2}\cdots x_{\pi(l_{s})}^{4s-2}.
\]
We want to relate the polynomial $\pi x^{\alpha_{\tau}}$ to $x^{\alpha
_{\pi\tau\pi^{-1}}}$. Observe that
\[
L_{1}^{\pi\tau\pi^{-1}}=\left\{  \pi(m_{1}),\ldots,\pi(m_{t})\right\}
\]
and in unordered form we have%
\[
\text{ }L_{2}^{\pi\tau\pi^{-1}}=\left\{  \pi(i_{1}),\pi(j_{1}),\ldots
,\pi(i_{r}),\pi(j_{r})\right\}  \text{, }Q_{2}^{\pi\tau\pi^{-1}}=\left\{
\pi(k_{1}),\pi(l_{1}),\ldots,\pi(k_{s}),\pi(l_{s})\right\}  .
\]
Define \
\[
\kappa=%
%TCIMACRO{\dprod \limits_{\pi\left(  i_{p}\right)  >\pi\left(  j_{p}\right)
%}}%
%BeginExpansion
{\displaystyle\prod\limits_{\pi\left(  i_{p}\right)  >\pi\left(  j_{p}\right)
}}
%EndExpansion
\left(  \pi\left(  i_{p}\right)  ,\pi\left(  j_{p}\right)  \right)  \cdot%
%TCIMACRO{\dprod \limits_{\pi\left(  k_{q}\right)  >\pi\left(  l_{q}\right)
%}}%
%BeginExpansion
{\displaystyle\prod\limits_{\pi\left(  k_{q}\right)  >\pi\left(  l_{q}\right)
}}
%EndExpansion
\left(  \pi\left(  k_{q}\right)  ,\pi\left(  l_{q}\right)  \right)  .
\]
Thus
\[
L_{2}^{\pi\tau\pi^{-1}}=\left\{  \kappa(\pi(i_{1})),\kappa(\pi(j_{1}%
)),\ldots,\kappa(\pi(i_{r})),\kappa(\pi(j_{r}))\right\}  \text{ and}%
\]%
\[
Q_{2}^{\pi\tau\pi^{-1}}=\left\{  \kappa(\pi(k_{1})),\kappa(\pi(l_{1}%
)),\ldots,\kappa(\pi(k_{s})),\kappa(\pi(l_{s}))\right\}  .
\]
Let $\varphi$ be the permutation that orders the sets
\[
\left\{  \kappa(\pi(i_{1})),\kappa(\pi(i_{2})),\ldots,\kappa(\pi
(i_{r})\right\}  \text{ and }\left\{  \kappa(\pi(k_{1})),\kappa(\pi
(k_{2})),\ldots,\kappa(\pi(k_{s}))\right\}
\]
so that
\[
\varphi(\kappa(\pi(i_{1})))<\varphi(\kappa(\pi(i_{2})))<\cdots<\varphi
(\kappa(\pi(i_{r}))\text{ and }\varphi(\kappa(\pi(k_{1})))<\varphi(\kappa
(\pi(k_{2})))<\cdots<\varphi(\kappa(\pi(k_{s}))).
\]
We let $\overline{\varphi}$ be the corresponding double permutation defined on
$L_{2}^{\pi\tau\pi^{-1}}\cup Q_{2}^{\pi\tau\pi^{-1}}$. Thus%
\[
\overline{\varphi}\kappa\pi x^{\alpha_{\tau}}=x^{\alpha_{\pi\tau\pi^{-1}}}.
\]
Observe that $sgn(\overline{\varphi})=$ $1$ and $sgn(\kappa)=\left(
-1\right)  ^{\varepsilon_{\pi}}$. Since $\kappa$ and $\overline{\varphi}$
belong to $\pi\vartheta_{\tau}\pi^{-1}$, using the formula in Remark 3.1, we
have
\[
\left(  -1\right)  ^{\varepsilon_{\pi}}P_{\pi\tau\pi^{-1}}=\left(  -1\right)
^{\varepsilon_{\pi}}\Omega_{\pi\vartheta_{\tau}\pi^{-1}}x^{\alpha_{\pi\tau
\pi^{-1}}}=sgn(\kappa)sgn(\overline{\varphi}^{-1})\Omega_{\pi\vartheta_{\tau
}\pi^{-1}}x^{\alpha_{\pi\tau\pi^{-1}}}=
\]%
\[
\Omega_{\pi\vartheta_{\tau}\pi^{-1}}\kappa\overline{\varphi}^{-1}%
x^{\alpha_{\pi\tau\pi^{-1}}}=\Omega_{\pi\vartheta_{\tau}\pi^{-1}}%
\kappa\overline{\varphi}^{-1}\overline{\varphi}\kappa\pi x^{\alpha_{\tau}%
}=\Omega_{\pi\vartheta_{\tau}\pi^{-1}}\pi x^{\alpha_{\tau}}=\pi P_{\tau}.
\]
Therefore $\phi(\pi e_{\tau})=\pi\phi\left(  e_{\tau}\right)  $.

Now it is simple matter to check that for a general element $\left(  \zeta
,\pi\right)  =\left(  \zeta,\boldsymbol{I}\right)  \left(  \boldsymbol{I}%
,\pi\right)  \in W$ we have
\[
\phi(\left(  \zeta,\pi\right)  e_{\tau})=\phi(\left(  \zeta,\boldsymbol{I}%
\right)  \left(  \boldsymbol{I},\pi\right)  e_{\tau})=\left(  -1\right)
^{\varepsilon_{\pi}}\phi\left(  \left(  \zeta,\boldsymbol{I}\right)
e_{\pi\tau\pi^{-1}}\right)  =
\]%
\[
\left(  -1\right)  ^{\varepsilon_{\pi}}\left(  \zeta,\boldsymbol{I}\right)
P_{\pi\tau\pi^{-1}}=\left(  \zeta,\boldsymbol{I}\right)  \pi P_{\tau}=\left(
\zeta,\pi\right)  \phi(e_{\tau})\text{.}%
\]

\end{proof}

\begin{definition}
\textbf{Telescopic decomposition\newline}Suppose $G$ is a group, $U$\ is a
$G$-module and $N$\ is a semisimple $G$-submodule of $U$. If $\partial
:U\rightarrow U$ is a nilpotent $G$-morphism then $\partial$ determines an
isomorphism of $N$ with a submodule of $U$ that decomposes as a direct sum
\[
N\cong\oplus_{k\geq0}N_{k}\subseteq U
\]
where%
\[
N_{k}=\left\{  m\in\partial^{k}\left(  N\right)  :\partial\left(  m\right)
=0\right\}  \text{.}%
\]
This $G$-submodule of $U$ will be called the telescopic decomposition of $N$
with respect to $\partial$ and will be denoted as:%
\[
\mathcal{D}_{\partial}\left(  N\right)  =\oplus_{k\geq0}N_{k}\text{.}%
\]
Note that%
\[
N\cong D_{\partial}\left(  N\right)  \subseteq U\text{.}%
\]

\end{definition}

\begin{example}
Consider the group $G=SO(n,\mathbb{R})$ of rotations in $\mathbb{R}^{n}$. We
have the natural action of $G$ on the space of polynomials $\mathcal{P}%
=\mathbb{C}\left[  x_{1},\ldots,x_{n}\right]  $. Let
\[
\partial=\sum_{i=1}^{n}\frac{\partial^{2}}{\partial x_{i}^{2}}%
\]
be the Laplacian and define
\[
U=%
%TCIMACRO{\dbigoplus \limits_{j=0}^{m}}%
%BeginExpansion
{\displaystyle\bigoplus\limits_{j=0}^{m}}
%EndExpansion
\mathcal{P}_{j}.
\]
Let $N=\mathcal{P}_{m}$ and let $\mathcal{H}_{j}$ indicate the subspace of
harmonic polynomials in $\mathcal{P}_{j}$. Then the telescopic decomposition
of $N$ with respect to $\partial$ is given by
\[
N_{k}=\mathcal{H}_{m-2k}\text{ so that }\mathcal{D}_{\partial}\left(
N\right)  =\mathcal{H}_{m}\oplus\mathcal{H}_{m-2}\oplus\mathcal{H}_{m-4}%
\oplus\cdots\text{ .}%
\]

\end{example}

For the groups $W=W(A_{n-1})$ or $W=W(B_{n})$ our basic strategy is as
follows. So far we have defined a surjective morphism of $W$-modules
\[
\phi:\mathcal{M}\rightarrow\mathfrak{M}\subseteq\mathcal{P}%
\]
where $\mathcal{M}$ is the Gelfand model from the previous section and where
$\mathfrak{\varphi}\left(  \mathcal{M}\right)  =\mathfrak{M}$ is the image of
$\phi$ in $\mathcal{P}$. We introduce the $W$-invariant differential operator
$\partial$ defined by
\[
\partial=\left\{
\begin{array}
[c]{ccc}%
\sum_{i=1}^{n}\frac{\partial}{\partial x_{i}} & \text{if} & W=W\left(
A_{n-1}\right) \\
\sum_{i=1}^{n}\frac{\partial^{2}}{\partial x_{i}^{2}} & \text{if} & W=W\left(
B_{n}\right)
\end{array}
\right.  .
\]
In what follows we will prove that the polynomial model $\mathcal{N}_{W}$ is
the telescopic decomposition of $\mathfrak{M}$ with respect to $\partial$. The
proof for $W=W(D_{2n+1})$ requires a slight modification and will be included
at the end of the article.

Observe that the symmetric group $\mathfrak{S}_{n}$ acts on the set of
multi-indices $M$ by
\[%
\begin{array}
[c]{ccc}%
\pi\cdot\alpha=\alpha\circ\pi^{-1} &  & \alpha\in M,\ \pi\in\mathfrak{S}_{n}%
\end{array}
.
\]

In the definition that follows, we assume that the elements of a subset
$I=\left\{  i_{1},\ldots,i_{k}\right\}  $ of $\mathbb{I}_{n}=\left\{
1,2,\ldots,n\right\}  $ are arranged in increasing order. Given a multi-index
$\alpha:I\rightarrow\mathbb{N}_{0}$ let $\left[  x_{i_{q}}^{\alpha(i_{p}%
)}\right]  $ be the $k\times k$ matrix with $x_{i_{q}}^{\alpha(i_{p})}$ in the
$(p,q)$-position and let $V\left(  \alpha,I\right)  $ denote the polynomial
defined by the determinant of the matrix $\left[  x_{i_{q}}^{\alpha(i_{p}%
)}\right]  $. Hence
\[
V\left(  \alpha,I\right)  =\left\vert
\begin{array}
[c]{cccc}%
x_{i_{1}}^{\alpha(i_{1})} & x_{i_{2}}^{\alpha(i_{1})} & \cdots & x_{i_{k}%
}^{\alpha(i_{1})}\\
x_{i_{1}}^{\alpha(i_{2})} & x_{i_{2}}^{\alpha(i_{2})} & \cdots & x_{i_{k}%
}^{\alpha(i_{2})}\\
\vdots & \vdots & \cdots & \vdots\\
x_{i_{1}}^{\alpha(i_{k})} & x_{i_{2}}^{\alpha(i_{k})} & \cdots & x_{i_{k}%
}^{\alpha(i_{k})}%
\end{array}
\right\vert .
\]
Observe that $V\left(  \alpha,I\right)  =0$ unless $\alpha:I\rightarrow
\mathbb{N}_{0}$ is injective. Also observe that when $\sigma$ is a permutation
of $I$ then
\[
V\left(  \sigma\cdot\alpha,I\right)  =sgn(\sigma)V\left(  \alpha,I\right)
=\sigma V\left(  \alpha,I\right)
\]
where $\sigma V\left(  \alpha,I\right)  $ indicates the previously defined
action of $\sigma$ on the polynomial $V\left(  \alpha,I\right)  $. Suppose
$V\left(  \alpha,I\right)  \neq0$ and let $\eta$ be the unique permutation of
$I$ such that
\[
\left(  \eta\cdot\alpha\right)  (i_{1})<\left(  \eta\cdot\alpha\right)
(i_{2})<\cdots<\left(  \eta\cdot\alpha\right)  (i_{k}).
\]
We say $V\left(  \alpha,I\right)  $ has \emph{positive orientation} if
$V\left(  \alpha,I\right)  =$ $V\left(  \eta\cdot\alpha,I\right)  $. Otherwise
$V\left(  \alpha,I\right)  $ has \emph{negative orientation}.

\begin{remark}
Suppose $I=\left\{  i_{1},\ldots,i_{k}\right\}  $, $J=\left\{  j_{1}%
,\ldots,j_{k}\right\}  $ are disjoint ordered subsets of indices such that
$i_{k}<j_{k}<i_{k+1}$ and let $F=I\cup J$ be the ordered set
\[
F=\left\{  i_{1},j_{1},i_{2},j_{2},\ldots,i_{k},j_{k}\right\}  .
\]
Suppose $\alpha:F\rightarrow\mathbb{N}_{0}$ is a multi-index and $\pi$ is a
permutation of $I$. Let $\overline{\pi}$ be the corresponding double
permutation. Then $V\left(  \alpha,F\right)  $ and $V\left(  \overline{\pi
}\cdot\alpha,F\right)  $ have the same orientation since the second
determinant is obtained from the first by exchanging two rows at a time.
\end{remark}

We now define the key notion of $W$\emph{-equivalence }on the set of
multi-indices $M$. Suppose $\alpha,\beta\in M$. For $W=W(A_{n-1})$ we say
$\alpha$ is $W$-\emph{equivalent to} $\beta$ if there exists a bijection
\[
\varphi:\alpha(\mathbb{I}_{n})\rightarrow\beta\left(  \mathbb{I}_{n}\right)
\]
such that $\varphi\circ\alpha=\beta$. In the case of $W=W\left(  B_{n}\right)
$ we also assume $\varphi$ sends even numbers to even numbers and odd numbers
to odd numbers, respectively. Write
\[
\alpha\sim_{W}\beta
\]
when $\alpha$ is equivalent to $\beta$. Observe that $W$-equivalence is
well-defined on the set of $\mathfrak{S}_{n}$-orbits in $M$. In particular,
for $\pi\in\mathfrak{S}_{n}$ if $\varphi\circ\alpha=\beta$ then $\varphi
\circ\left(  \pi\cdot\alpha\right)  =\pi\cdot\beta$. Note if $\gamma$ is a
$\mathfrak{S}_{n}$-orbit in $M$ then the value
\[
\left\vert \alpha\right\vert =%
%TCIMACRO{\dsum \nolimits_{i\in\mathbb{I}_{n}}}%
%BeginExpansion
{\displaystyle\sum\nolimits_{i\in\mathbb{I}_{n}}}
%EndExpansion
\alpha(i)
\]
is independent of the choice of $\alpha$, hence we can speak of the value
$\left\vert \gamma\right\vert $. A $\mathfrak{S}_{n}$-orbit $\gamma$ will be
called $W$\emph{-minimal} if
\[
\rho\sim_{W}\gamma\Rightarrow\left\vert \gamma\right\vert \leq\left\vert
\rho\right\vert \text{.}%
\]

Given a $W$-minimal orbit $\gamma$ in $M$ we define a corresponding
$W$-submodule $S_{\gamma}$ of $\mathcal{P}$ given by the span of the set of
polynomials $\left\{  x^{\alpha}:\alpha\in\gamma\right\}  $.

\begin{example}
Consider the group $W=W(B_{3})$. In this case $\mathcal{P}=\mathbb{C}\left[
x,y,z\right]  $ and a multi-index is a triple $\left(  a,b,c\right)  $ of
non-negative integers representing the polynomial $x^{a}y^{b}z^{c}$. Write
$\left[  a,b,c\right]  $ for the corresponding $\mathfrak{S}_{3}$-orbit. Then
there are 10 different $W$-minimal orbits: $\left[  0,0,0\right]  $, $\left[
1,1,1\right]  $, $\left[  0,0,1\right]  $, $\left[  0,0,2\right]  $, $\left[
1,1,0\right]  ,$ $\left[  1,1,3\right]  $, $\left[  0,2,4\right]  $, $\left[
0,2,1\right]  $, $\left[  0,1,3\right]  $ and $\left[  1,3,5\right]  $.
\end{example}

Let $S_{\gamma}^{0}$ be the $W$-submodule of $S_{\gamma}$ defined by%

\[
S_{\gamma}^{0}=\left\{  P\in S_{\gamma}:\partial\left(  P\right)  =0\right\}
.
\]

\begin{theorem}
\label{3}If $\gamma$ is $W$-minimal then representation $S_{\gamma}^{0}$ is
irreducible and
\[
\mathcal{N}_{W}=%
%TCIMACRO{\dbigoplus \limits_{\gamma\text{ }W\text{-minimal}}}%
%BeginExpansion
{\displaystyle\bigoplus\limits_{\gamma\text{ }W\text{-minimal}}}
%EndExpansion
S_{\gamma}^{0}\text{.}%
\]

\end{theorem}

\begin{proof}
In the case of $W=W(A_{n-1})$ the theorem follows from \cite[Corollary 2.3 and
Theorem 4,2]{aguado1}. In the case of $W=W(B_{n})$ the result follows from
\cite[Corollary 2.4 and Theorem 4.4]{araujo1}. 
\end{proof}

\bigskip

After some initial preparation we will show when $\gamma$ is $W$-minimal then
$S_{\gamma}^{0}\cap\mathcal{D}_{\partial}\left(  \mathfrak{M}\right)
\neq\left\{  0\right\}  $. Let $I=\left\{  i_{1},i_{2},\ldots,i_{k}\right\}
\subseteq\mathbb{I}_{n}$ and define multi-indices $\iota,$ $\varphi$ and
$\kappa:I\rightarrow\mathbb{N}_{0}$ by
\[%
\begin{array}
[c]{cccc}%
\iota(i_{j})=j-1, & \varphi(i_{j})=2\left(  j-1\right)  & \text{and} &
\kappa(i_{j})=2j-1
\end{array}
.
\]
We introduce the groups $\mathfrak{S}\left(  I\right)  $ and $W\left(
B_{I}\right)  =\mathcal{C}_{2}^{k}\rtimes\mathfrak{S}\left(  I\right)  $. Note
that $\iota$ is $\mathfrak{S}(I)$-minimal and that $\varphi$ and $\kappa$ are
$W\left(  B_{I}\right)  $-minimal. We also introduce the restriction
$\partial_{I}$ of $\partial$ to $\mathbb{C}\left[  x_{i_{1}},\ldots,x_{i_{k}%
}\right]  $ given by
\[
\partial_{I}=\left\{
\begin{array}
[c]{ccc}%
\sum_{j=1}^{k}\frac{\partial}{\partial x_{i_{j}}} & \text{if} & W=W\left(
A_{n-1}\right) \\
\sum_{j=1}^{n}\frac{\partial^{2}}{\partial x_{i_{j}}^{2}} & \text{if} &
W=W\left(  B_{n}\right)
\end{array}
\right.  .
\]

\begin{lemma}
\label{4}i) Suppose $\alpha\in M$ and, if $W=W\left(  B_{n}\right)  $,
$\alpha$ takes the same parity on $I$. In that case, when $V\left(
\alpha,I\right)  $ has positive orientation then a power of $\partial_{I}$
transforms $V\left(  \alpha,I\right)  $ in a positive scalar multiple of
either $V\left(  \iota,I\right)  $, $V\left(  \varphi,I\right)  $ or $V\left(
\kappa,I\right)  $.

ii) $\partial_{I}\left(  V\left(  \iota,I\right)  \right)  =0$ and if
$W=W\left(  B_{n}\right)  $ then $\partial_{I}\left(  V\left(  \kappa
,I\right)  \right)  =0=\partial_{I}\left(  V\left(  \kappa,I\right)  \right)
$.
\end{lemma}

\begin{proof}
\emph{i)}\ Assume $V\left(  \alpha,I\right)  $ has positive orientation
\ Write $\alpha^{p}=\left(  \alpha(i_{1}),\ldots,\alpha(i_{p})-2,\ldots
,\alpha(i_{k})\right)  $ in the case $W=W\left(  B_{n}\right)  $ and
$\alpha^{p}=\left(  \alpha(i_{1}),\ldots,\alpha(i_{p})-1,\ldots,\alpha
(i_{k})\right)  $ in the case $W=\mathfrak{S}_{n}$. Since $\partial_{I}$ is
$W_{I}$-invariant, we have
\begin{align*}
\partial_{I}\left(  V\left(  \alpha,I\right)  \right)   &  =\partial_{I}%
\sum_{\pi\in\mathfrak{S}(I)}sg\left(  \pi\right)  \pi x^{\alpha}=\sum_{\pi
\in\mathfrak{S}(I)}sg\left(  \pi\right)  \pi\partial_{I}\left(  x^{\alpha
}\right) \\
&  =\sum_{\pi\in\mathfrak{S}(I)}sg\left(  \pi\right)  \pi\left(  \sum
_{p=1}^{k}\lambda_{p}x^{\alpha^{p}}\right) \\
&  =\sum_{p=1}^{k}\lambda_{p}\sum_{\pi\in\mathfrak{S}(I)}sg\left(  \pi\right)
\pi x^{\alpha^{p}}\\
&  =\sum_{p=1}^{k}\lambda_{p}V\left(  \alpha^{p},I\right)
\end{align*}
where $\lambda_{p}=\alpha(i_{p})(\left(  \alpha(i_{p})-1\right)  $ if
$W=W\left(  B_{n}\right)  $ or $\lambda_{p}=\alpha(i_{p})$ if $W=\mathfrak{S}%
_{n}$. In the terms where $V\left(  \alpha^{p},I\right)  \neq0$, is clear that
$V\left(  \alpha^{p},I\right)  $ has positive orientation and $\left\vert
\alpha^{p}\right\vert =\left\vert \alpha\right\vert -1$ if $W=\mathfrak{S}%
_{n}$ or $\left\vert \alpha^{p}\right\vert =\left\vert \alpha\right\vert -2$
if $W=W\left(  B_{n}\right)  $. Now the proof follows by induction on the
$\left\vert \alpha\right\vert =\sum_{j=1}^{k}\alpha\left(  i_{j}\right)  $.

\emph{ii) }By \emph{i)} we have:%
\[
\partial_{I}\left(  V\left(  \alpha,I\right)  \right)  =\sum_{p=1}^{k}%
\lambda_{p}V\left(  \alpha^{p},I\right)
\]
but $\alpha^{p}$ is not injective when $\alpha=\iota,$ $\varphi$ or $\kappa$
so $V\left(  \alpha^{p},I_{p}\right)  =0$, $1\leq p\leq k$, in these cases.
\end{proof}

\begin{corollary}
\label{5}Let $\mathbb{I}_{n}=\cup_{j=1}^{h}I_{j}$ be a partition of
$\mathbb{I}_{n}$ and suppose $\alpha\in M$ is injective. Let $\alpha_{j}$
denote the restriction of $\alpha$ to the set $I_{j}$. Assume (as before) that
$\alpha_{j}$ takes the same parity on $I_{j}$ if $W=W\left(  B_{n}\right)  $.
Then there exists $m\in\mathbb{N}_{0}$ such that:%
\[
\partial^{m}\left(
%TCIMACRO{\dprod \limits_{j=1}^{h}}%
%BeginExpansion
{\displaystyle\prod\limits_{j=1}^{h}}
%EndExpansion
V\left(  \alpha_{j},I_{j}\right)  \right)  =q%
%TCIMACRO{\dprod \limits_{j=1}^{h}}%
%BeginExpansion
{\displaystyle\prod\limits_{j=1}^{h}}
%EndExpansion
V\left(  \beta_{j},I_{j}\right)
\]
where $\beta_{j}=\iota,\varphi$ or $\kappa$ on $I_{j}$ as previously defined
and where $q\in\mathbb{N}$.
\end{corollary}

\begin{proof}
The operator $\partial$ can be decomposed as $\partial=\sum_{j=1}^{h}%
\partial_{j}\,$with%
\[
\partial_{j}=\left\{
\begin{array}
[c]{ccc}%
\sum_{k\in I_{j}}\frac{\partial}{\partial x_{k}} & \text{if} & W=\mathfrak{S}%
_{n}\\
\sum_{k\in I_{j}}\frac{\partial^{2}}{\partial x_{k}^{2}} & \text{if} &
W=W\left(  B_{n}\right)
\end{array}
\right.  .
\]
From the previous lemma there exists $m_{j}\in\mathbb{N}_{0}$ such that
$\partial_{j}^{m_{j}}V\left(  \alpha_{j},I_{j}\right)  =q_{j}V\left(
\beta_{j},I_{j}\right)  $ where $\beta_{j}=\iota,\varphi$ or $\kappa$ as
defined on $I_{j}$ and $q_{j}\in\mathbb{N}$. Let $m=\sum_{j=1}^{h}m_{j}$. Then
we have%
\[
\partial^{m}=\left(  \sum_{j=1}^{h}\partial_{j}\right)  ^{m}=\sum_{\left\vert
\delta\right\vert =m}c_{\delta}%
%TCIMACRO{\dprod \limits_{j=1}^{h}}%
%BeginExpansion
{\displaystyle\prod\limits_{j=1}^{h}}
%EndExpansion
\partial_{j}^{\delta_{j}}%
\]
so that
\[
\partial^{m}\left(
%TCIMACRO{\dprod \limits_{j=1}^{h}}%
%BeginExpansion
{\displaystyle\prod\limits_{j=1}^{h}}
%EndExpansion
V\left(  \alpha_{j},I_{j}\right)  \right)  =\sum_{\left\vert \delta\right\vert
=m}c_{\delta}\left(
%TCIMACRO{\dprod \limits_{j=1}^{h}}%
%BeginExpansion
{\displaystyle\prod\limits_{j=1}^{h}}
%EndExpansion
\partial_{j}^{\delta_{j}}V\left(  \beta_{j},I_{j}\right)  \right)  .
\]
Take $\delta$ such that $\left\vert \delta\right\vert =m$. If $\delta
_{j}>m_{j}$ for some index $j$, then $\partial_{j}^{\delta_{j}}V\left(
\alpha_{j},I_{j}\right)  =0$. If $\delta_{j}<m_{j}$ for some index $j$ then
there exists $i$ such that $\delta_{i}>m_{i}$ and $\partial_{i}^{\delta_{i}%
}V\left(  \alpha_{i},I_{i}\right)  =0$. It follows that the only non-zero term
in the in the right-hand side of the previous equality corresponds to
$\delta=\left(  m_{1},\ldots,m_{h}\right)  $. Hence
\[
\partial^{m}\left(
%TCIMACRO{\dprod \limits_{j=1}^{h}}%
%BeginExpansion
{\displaystyle\prod\limits_{j=1}^{h}}
%EndExpansion
V\left(  \alpha_{j},I_{j}\right)  \right)  =q\left(
%TCIMACRO{\dprod \limits_{j=1}^{h}}%
%BeginExpansion
{\displaystyle\prod\limits_{j=1}^{h}}
%EndExpansion
V\left(  \beta_{j},I_{j}\right)  \right)  .
\]

\end{proof}

\begin{lemma}
\label{6}Let $\gamma$ be a $W$-minimal orbit in $M$. Then $S_{\gamma}%
^{0}\subseteq\mathcal{D}_{\partial}\left(  \mathfrak{M}\right)  $.
\end{lemma}

\begin{proof}
The idea of the proof is to produce a nontrivial element in $S_{\gamma}%
^{0}\cap\mathcal{D}_{\partial}\left(  \mathfrak{M}\right)  $. Suppose
$\alpha\in\gamma$. Observe that there is a decomposition of $\mathbb{I}%
_{n}=\cup_{j=1}^{h}I_{j}$ such that:

\emph{i)} The restrictions $\alpha_{j}$ of $\alpha$ to $I_{j}$ are injective;

\emph{ii)} If $W=W\left(  B_{n}\right)  $, $\alpha_{j}$ takes the same parity
on $I_{j}$;

\emph{iii)} Each $I_{j}$ is maximal with the properties \emph{i) }and
\emph{ii)}.

\noindent\noindent Since $\gamma$ is a $W$-minimal orbit, it follows that
every $\alpha_{j}$ is of the form $\iota$, $\varphi$ or $\kappa$, so by Lemma
\ref{4} and Corollary \ref{5} we have
\[%
%TCIMACRO{\dprod \limits_{j=1}^{h}}%
%BeginExpansion
{\displaystyle\prod\limits_{j=1}^{h}}
%EndExpansion
V\left(  \alpha_{j},I_{j}\right)  \in S_{\gamma}^{0}.
\]

We will show that the polynomial $%
%TCIMACRO{\dprod \limits_{j=1}^{h}}%
%BeginExpansion
{\displaystyle\prod\limits_{j=1}^{h}}
%EndExpansion
V\left(  \alpha_{j},I_{j}\right)  $ is in the telescopic decomposition of
$\mathfrak{M}$. Since $%
%TCIMACRO{\dprod \limits_{j=1}^{h}}%
%BeginExpansion
{\displaystyle\prod\limits_{j=1}^{h}}
%EndExpansion
V\left(  \alpha_{j},I_{j}\right)  $ is in the space $S_{\gamma}^{0}$ this
means it is in the kernel of $\partial$. Therefore it suffices to prove that
the polynomial in question is in the image of a power of $\partial$ applied to
$\mathfrak{M}$. Put $I_{j}=\left\{  i_{1},i_{2},\ldots,i_{m_{j}}\right\}  $.
When $W\cong\mathfrak{S}_{n}$ define
\begin{align*}
\tau_{j}  &  =\left\{
\begin{array}
[c]{ccc}%
\left(  i_{1},i_{2}\right)  \cdots\left(  i_{m_{j}-1},i_{m_{j}}\right)  &
\text{if} & m_{j}\text{ is even}\\
\left(  i_{1},i_{2}\right)  \cdots\left(  i_{m_{j}-2},i_{m_{j}-1}\right)  &
\text{if} & m_{j}>1\text{ is odd}%
\end{array}
\right. \\
&  \text{and}\\
&
\begin{array}
[c]{ccc}%
\tau_{j}=\boldsymbol{I} & \text{if} & m_{j}=1
\end{array}
.
\end{align*}
When $W=W(B_{n})$ we first define the element $\zeta^{I_{j}}\in\mathcal{C}%
_{2}^{n}$ by
\[
\zeta_{k}^{I_{j}}=\left\{
\begin{array}
[c]{ccc}%
-1 & \text{if} & k\in I_{j}\\
1 & \text{if} & k\notin I_{j}%
\end{array}
\right.  .
\]
Then define
\begin{align*}
\tau_{j}  &  =\left\{
\begin{array}
[c]{ccc}%
(\boldsymbol{I},\left(  i_{1},i_{2}\right)  \cdots\left(  i_{m_{j}-1}%
,i_{m_{j}}\right)  ) & \text{if} & m_{j}\text{ is even and }\alpha\text{ is
even on }I_{j}\\
\left(  \zeta^{I_{j}},\left(  i_{1},i_{2}\right)  \cdots\left(  i_{m_{j}%
-1},i_{m_{j}}\right)  \right)  & \text{if} & m_{j}\text{ is even and }%
\alpha\text{ is odd on }I_{j}\\
\left(  \boldsymbol{I},\left(  i_{1},i_{2}\right)  \cdots\left(  i_{m_{j}%
-2},i_{m_{j}-1}\right)  \right)  & \text{if} & m_{j}>1\text{ is odd and
}\alpha\text{ is even on }I_{j}\\
\left(  \zeta^{I_{j}},\left(  i_{1},i_{2}\right)  \cdots\left(  i_{m_{j}%
-2},i_{m_{j}-1}\right)  \right)  & \text{if} & m_{j}>1\text{ is odd and
}\alpha\text{ is odd on }I_{j}%
\end{array}
\right. \\
&  \text{and}\\
\tau_{j}  &  =\left\{
\begin{array}
[c]{ccc}%
\boldsymbol{I} & \text{if} & m_{j}=1\text{ and }\alpha\text{ is even on }%
I_{j}\\
\left(  \zeta^{I_{j}},\boldsymbol{I}\right)  & \text{if} & m_{j}=1\text{ and
}\alpha\text{ is odd on }I_{j}%
\end{array}
\right.  .
\end{align*}
Let $\tau\in W$ be the product of the involutions $\tau_{j}$ and let
$\vartheta_{j}$ be the centralizer of $\tau_{j}$ in $\mathfrak{S}%
_{j}=\mathfrak{S}\left(  I_{j}\right)  $. Suppose $\mathcal{T}_{j}$ is a
system of representatives of right cosets of $\vartheta_{j}$ in $\mathfrak{S}%
_{j}$. We put
\begin{align*}
\mathfrak{S}_{\gamma}  &  =\mathfrak{S}_{1}\times\cdots\times\mathfrak{S}%
_{m}\\
\vartheta_{\gamma}  &  =\vartheta_{1}\times\cdots\times\vartheta_{m}\\
\mathcal{T}_{\gamma}  &  =\mathcal{T}_{1}\times\cdots\times\mathcal{T}%
_{m}\text{.}%
\end{align*}
Let $\vartheta_{\tau}=\vartheta_{\tau}^{+}\times\vartheta_{\tau}^{-}$ as
before and let $\mathfrak{H}$ denote a set of double permutations that give a
system of representatives of left cosets of $\vartheta_{\gamma}$ in
$\vartheta_{\tau}$. According to the definition of $P_{\tau}$ we have
\[
\Omega_{\mathcal{T}_{\gamma}}P_{\tau}=\Omega_{\mathcal{T}_{\gamma}}%
\Omega_{\mathcal{\vartheta}_{\tau}}x^{\alpha_{\tau}}\text{.}%
\]
Decomposing
\[
\Omega_{\mathcal{\vartheta}_{\tau}}=\sum_{\eta\in\mathfrak{H}}\Omega
_{\mathcal{\vartheta}_{\gamma}}\eta
\]
we have%
\begin{align*}
\Omega_{\mathcal{T}_{\gamma}}P_{\tau}  &  =\Omega_{\mathcal{T}_{\gamma}%
}\left(  \sum_{\eta\in\mathfrak{H}}\Omega_{\mathcal{\vartheta}_{\gamma}}%
\eta\right)  x^{\alpha_{\tau}}\\
&  =\sum_{\eta\in\mathfrak{H}}\Omega_{\mathfrak{S}_{\gamma}}\eta
x^{\alpha_{\tau}}\\
&  =\sum_{\eta\in\mathfrak{H}}%
%TCIMACRO{\dprod \limits_{j=1}^{n}}%
%BeginExpansion
{\displaystyle\prod\limits_{j=1}^{n}}
%EndExpansion
V\left(  \left(  \eta\cdot\alpha_{\tau}\right)  _{j},I_{j}\right)  \text{.}%
\end{align*}
Since the permutations $\eta$ are double permutations all factors $V\left(
\left(  \eta\cdot\alpha_{\tau}\right)  _{j},I_{j}\right)  $ have positive
orientation. Now, by Corollary \ref{5}, there is a power $m$ of $\partial$
such that%
\[
\partial^{m}\left(  \sum_{\eta\in\mathfrak{H}}%
%TCIMACRO{\dprod \limits_{j=1}^{n}}%
%BeginExpansion
{\displaystyle\prod\limits_{j=1}^{n}}
%EndExpansion
V\left(  \left(  \eta\cdot\delta_{\tau}\right)  _{j},I_{j}\right)  \right)  =q%
%TCIMACRO{\dprod \limits_{j=1}^{h}}%
%BeginExpansion
{\displaystyle\prod\limits_{j=1}^{h}}
%EndExpansion
V\left(  \alpha_{j},I_{j}\right)
\]
where $q\in\mathbb{N}$.
\end{proof}

\begin{theorem}
Suppose $W=W\left(  A_{n-1}\right)  $ or $W=W\left(  B_{n}\right)  $ and
$\mathcal{M}$ is the Gelfand model defined in the previous section. Let
$\phi:\mathcal{M}\rightarrow\mathcal{P}=\mathbb{C}\left[  x_{1},\ldots
,x_{n}\right]  $ be the $W$-morphism defined by
\[
e_{\tau}\mapsto P_{\tau}%
\]
and put $\mathfrak{M}=\phi\left(  \mathcal{M}\right)  $. Let
\[
\partial:\mathcal{P}\rightarrow\mathcal{P}%
\]
be the $W$-invariant differential operator previously defined in this section.
Then $\mathcal{N}_{W}=\mathcal{D}_{\partial}\left(  \mathfrak{M}\right)  $,
that is: the polynomial model for $W$ is the telescopic decomposition of
$\mathfrak{M}$ with respect to $\partial$. In particular we have isomorphisms
\[
\mathcal{M}\overset{\phi}{\cong}\mathfrak{M}\cong\mathcal{D}_{\partial}\left(
\mathfrak{M}\right)  =\mathcal{N}_{W}\text{.}%
\]

\end{theorem}

\begin{proof}
By the Lemma \ref{6} it follows that $S_{\gamma}^{0}\subseteq\mathfrak{M}$ for
every $W$-minimal $\gamma\in M$. From Theorem \ref{3} it follows that
$\mathcal{N}_{W}\subseteq\mathfrak{M=\phi}\left(  \mathcal{M}\right)  $. But
$\dim_{\mathbb{C}}\left(  \mathcal{N}_{W}\right)  =\dim_{\mathbb{C}}\left(
\mathcal{M}\right)  $ because both are Gelfand models.
\end{proof}

\bigskip

We now consider a simple modification of the above construction that gives an
isomorphism for the group $W\left(  D_{2n+1}\right)  $. The trick is to
replace some of the spaces $\mathcal{M}_{\tau}$ for $\tau\in\mathcal{I}$ with
the spaces $\mathcal{M}_{-\tau}$ where $-\tau$ is the corresponding involution
in $W\left(  B_{2n+1}\right)  $ (notation as in Remark 2.8 from the previous
section) and then apply the telescopic decomposition as in the case of
$W\left(  B_{2n+1}\right)  $. First we establish the following.

\begin{proposition}
Suppose $W=W\left(  D_{n}\right)  $ and let $\tau\in\mathcal{I}$. Then
$\mathcal{M}_{\tau}$ and $\mathcal{M}_{-\tau}$ are isomorphic $W$-modules.
\end{proposition}

\begin{proof}
Let $O$ denote the $\mathfrak{S}_{n}$-orbit of $\tau$ in $W\left(
D_{n}\right)  $. Thus the vectors $e_{\mu}$ for $\mu\in O$ are a basis of
$\mathcal{M}_{\tau}$ and the linear map given by
\[
\theta:\mathcal{M}_{\tau}\rightarrow\mathcal{M}_{-\tau}\text{ given by }%
\theta(e_{\mu})=e_{-\mu}%
\]
is an isomorphism of vector spaces. We show $\theta$ is a morphism of
$W\left(  D_{n}\right)  $-modules. Suppose $\pi\in\mathfrak{S}_{n}$ Then we
have
\[
\theta\left(  \pi e_{\tau}\right)  =\theta\left(  e_{\pi\tau\pi^{-1}}\right)
=e_{-\pi\tau\pi^{-1}}=e_{\pi\left(  -\tau\right)  \pi^{-1}}=\pi e_{-\tau}%
=\pi\left(  \theta e_{\tau}\right)  .
\]
Now suppose $\sigma=\left(  \zeta,I\right)  $ with $\zeta\in\mathcal{C}%
_{2}^{n}$ such that $%
%TCIMACRO{\dprod \limits_{i=1}^{n}}%
%BeginExpansion
{\displaystyle\prod\limits_{i=1}^{n}}
%EndExpansion
\zeta_{i}=1$. Observe that
\[
L_{1}^{-\tau}=Q_{1}^{\tau}\text{ and }L_{2}^{-\tau}=Q_{2}^{\tau}.
\]
Since
\[
1=%
%TCIMACRO{\dprod \limits_{i=1}^{n}}%
%BeginExpansion
{\displaystyle\prod\limits_{i=1}^{n}}
%EndExpansion
\zeta_{i}=\left(
%TCIMACRO{\dprod \limits_{i\in L_{1}^{\tau}\cup L_{2}^{\tau}}}%
%BeginExpansion
{\displaystyle\prod\limits_{i\in L_{1}^{\tau}\cup L_{2}^{\tau}}}
%EndExpansion
\zeta_{i}\right)  \left(
%TCIMACRO{\dprod \limits_{i\in Q_{1}^{\tau}\cup Q_{2}^{\tau}}}%
%BeginExpansion
{\displaystyle\prod\limits_{i\in Q_{1}^{\tau}\cup Q_{2}^{\tau}}}
%EndExpansion
\zeta_{i}\right)
\]
it follows that
\[
\left(
%TCIMACRO{\dprod \limits_{i\in L_{1}^{\tau}\cup L_{2}^{\tau}}}%
%BeginExpansion
{\displaystyle\prod\limits_{i\in L_{1}^{\tau}\cup L_{2}^{\tau}}}
%EndExpansion
\zeta_{i}\right)  =\left(
%TCIMACRO{\dprod \limits_{i\in L_{1}^{-\tau}\cup L_{2}^{-\tau}}}%
%BeginExpansion
{\displaystyle\prod\limits_{i\in L_{1}^{-\tau}\cup L_{2}^{-\tau}}}
%EndExpansion
\zeta_{i}\right)  .
\]
Therefore
\begin{align*}
\theta\left(  \sigma e_{\tau}\right)   &  =\theta\left(  \left(
%TCIMACRO{\dprod \limits_{i\in L_{1}^{\tau}\cup L_{2}^{\tau}}}%
%BeginExpansion
{\displaystyle\prod\limits_{i\in L_{1}^{\tau}\cup L_{2}^{\tau}}}
%EndExpansion
\zeta_{i}\right)  e_{\tau}\right)  =\left(
%TCIMACRO{\dprod \limits_{i\in L_{1}^{\tau}\cup L_{2}^{\tau}}}%
%BeginExpansion
{\displaystyle\prod\limits_{i\in L_{1}^{\tau}\cup L_{2}^{\tau}}}
%EndExpansion
\zeta_{i}\right)  \theta\left(  e_{\tau}\right)  =\\
&  =\left(
%TCIMACRO{\dprod \limits_{i\in L_{1}^{\tau}\cup L_{2}^{\tau}}}%
%BeginExpansion
{\displaystyle\prod\limits_{i\in L_{1}^{\tau}\cup L_{2}^{\tau}}}
%EndExpansion
\zeta_{i}\right)  e_{-\tau}=\left(
%TCIMACRO{\dprod \limits_{i\in L_{1}^{-\tau}\cup L_{2}^{-\tau}}}%
%BeginExpansion
{\displaystyle\prod\limits_{i\in L_{1}^{-\tau}\cup L_{2}^{-\tau}}}
%EndExpansion
\zeta_{i}\right)  e_{-\tau}=\sigma\left(  \theta e_{\tau}\right)  \text{.}%
\end{align*}
Hence it follows, as in the proof of Proposition 3.4, the $\theta$ is a
morphism of $W\left(  D_{n}\right)  $-modules$.$
\end{proof}

\bigskip

In order to construct a specific isomorphism between the Gelfand module
$\mathcal{M}$ and the polynomial module for the group $W\left(  D_{2n+1}%
\right)  $ we can now proceed as follows. Using the notation from the previous
section, to each involution $\tau\in\mathcal{I}$ we define a vector $e_{\tau
}^{\vee}$ in the symmetric algebra in the following way:
\[
e_{\tau}^{\vee}=\left\{
\begin{array}
[c]{ccc}%
e_{\tau} & \text{if} & \left\vert Q_{1}^{\tau}\cup Q_{2}^{\tau}\right\vert
>\left\vert L_{1}^{\tau}\cup L_{2}^{\tau}\right\vert \\
e_{-\tau} & \text{if} & \left\vert L_{1}^{\tau}\cup L_{2}^{\tau}\right\vert
>\left\vert Q_{1}^{\tau}\cup Q_{2}^{\tau}\right\vert
\end{array}
\right.  .
\]
We let $\mathcal{M}_{\tau}^{\vee}$ denote the $W$-module generated by
$\mathfrak{S}_{n}$-orbit de $e_{\tau}^{\vee}$ and define
\[
\mathcal{M}^{\vee}=%
%TCIMACRO{\dbigoplus \limits_{\tau\in\mathfrak{R}}}%
%BeginExpansion
{\displaystyle\bigoplus\limits_{\tau\in\mathfrak{R}}}
%EndExpansion
\mathcal{M}_{\tau}^{\vee}%
\]
where $\mathfrak{R}$ is a system of representatives of the $\mathfrak{S}_{n}%
$-orbits in $\mathcal{I}$.

The following result follows immediately from the above considerations.

\begin{corollary}
If $W=W\left(  D_{n}\right)  $ then module $\mathcal{M}^{\vee}$ is isomorphic
to the module $\mathcal{M}$ constructed in the previous section. In particular
$\mathcal{M}^{\vee}$ is a Gelfand module for $W=W\left(  D_{2n+1}\right)  $.
\end{corollary}

To each of the vectors $e_{\tau}^{\vee}$ we associate the polynomial
\[
P_{\tau}^{\vee}=\left\{
\begin{array}
[c]{ccc}%
P_{\tau} & \text{if} & e_{\tau}^{\vee}=e_{\tau}\\
P_{-\tau} & \text{if} & e_{\tau}^{\vee}=e_{-\tau}%
\end{array}
\right.  .
\]
and let $\mathfrak{M}^{\vee}$ be the corresponding $W$-submodule of
$\mathcal{P}$ spanned by the polynomials $P_{\tau}^{\vee}$ for $\tau
\in\mathcal{I}$. As before, the linear application
\[
e_{\tau}^{\vee}\mapsto P_{\tau}^{\vee}%
\]
defines a surjection of $W$-modules To proceed with the telescopic
decomposition we utilize the Laplacian, as in the case of $W\left(
B_{2n+1}\right)  $. The advantage of using the module $\mathcal{M}^{\vee}$ is
the following.

\begin{theorem}
The telescopic decomposition of $\mathfrak{M}^{\vee}$ is the polynomial model
for $W\left(  D_{2n+1}\right)  $.
\end{theorem}

\begin{proof}
Let $\mathcal{M}_{B_{2n+1}}$ denote the Gelfand module for $W\left(
B_{2n+1}\right)  $ from the previous section and let $\mathfrak{M}_{B_{2n+1}}$
denote the corresponding submodule of $\mathcal{P}$ defined in this section.
We have a natural inclusion
\[
\mathfrak{M}^{\vee}\subseteq\mathfrak{M}_{B_{2n+1}}\text{. }%
\]
We will use the branching rules for the polynomial model established in
\cite{araujo2}. Suppose $\gamma$ is a $W\left(  B_{2n+1}\right)  $-equivalence
class in the set of multi-indices. Then the cardinalities of the following two
sets are independent of the choice of $\alpha\in\gamma$:
\[
\text{\ }\gamma_{E}=\left\vert \left\{  i\in\mathbb{I}_{2n+1}:\alpha(i)\text{
is even}\right\}  \right\vert \text{ \ and }.\gamma_{O}=\left\vert \left\{
i\in\mathbb{I}_{2n+1}:\alpha(i)\text{ is odd}\right\}  \right\vert .
\]
We say that $\gamma$ is $W\left(  D_{2n+1}\right)  $-\emph{minimal} if
$\gamma$ is $W\left(  B_{2n+1}\right)  $-minimal and if the cardinalities
$\gamma_{E}$ and $\gamma_{O}$ satisfy
\[
\gamma_{E}>\gamma_{O}.
\]
Let $\Delta$ denote the Laplacian. The polynomial model for $W\left(
D_{2n+1}\right)  $ is
\[
\mathcal{N}_{W\left(  D_{2n+1}\right)  }=%
%TCIMACRO{\dbigoplus \limits_{\gamma\text{ }W\left(  D_{2n+1}\right)
%\text{-minimal}}}%
%BeginExpansion
{\displaystyle\bigoplus\limits_{\gamma\text{ }W\left(  D_{2n+1}\right)
\text{-minimal}}}
%EndExpansion
S_{\gamma}^{0}\subseteq\mathcal{D}_{\Delta}\left(  \mathfrak{M}_{B_{2n+1}%
}\right)
\]
where $S_{\gamma}^{0}$ is defined exactly as in the case of $W\left(
B_{2n+1}\right)  $ \cite[Theorem 18]{araujo2}. By our construction it follows
that the polynomials $P_{\tau}^{\vee}$ are sums of monomials that have more
even exponents than odd ones, which is also true of the monomials in the
$W\left(  D_{2n+1}\right)  $-module $\mathfrak{M}_{\tau}^{\vee}$ (where this
denotes the module generated by $P_{\tau}^{\vee}$). On the one hand we have
\[
\mathcal{D}_{\Delta}\left(  \mathfrak{M}_{\tau}^{\vee}\right)  \subseteq%
%TCIMACRO{\dbigoplus \limits_{\gamma\text{ }W\left(  B_{2n+1}\right)
%\text{-minimal}}}%
%BeginExpansion
{\displaystyle\bigoplus\limits_{\gamma\text{ }W\left(  B_{2n+1}\right)
\text{-minimal}}}
%EndExpansion
S_{\gamma}^{0}.
\]
However, when a monomial $x^{\alpha}$ has more even exponents than odd ones,
then this is also true of the monomials in the expansion of $\Delta(x^{\alpha
})$. Therefore
\[
\mathcal{D}_{\Delta}\left(  \mathfrak{M}_{\tau}^{\vee}\right)  \subseteq%
%TCIMACRO{\dbigoplus \limits_{\gamma\text{ }W\left(  D_{2n+1}\right)
%\text{-minimal}}}%
%BeginExpansion
{\displaystyle\bigoplus\limits_{\gamma\text{ }W\left(  D_{2n+1}\right)
\text{-minimal}}}
%EndExpansion
S_{\gamma}^{0}.
\]
Hence
\[
\mathcal{D}_{\Delta}\left(  \mathfrak{M}^{\vee}\right)  \subseteq
\mathcal{N}_{W\left(  D_{2n+1}\right)  }.
\]
Thus they must coincide, because they are both Gelfand models.
\end{proof}

\begin{example}
Recall the 10 different $W\left(  B_{3}\right)  $-minimal orbits
characterizing the polynomial model in Example 3.8. The corresponding
$W\left(  D_{3}\right)  $-minimal orbits are $\left[  0,0,0\right]  $,
$\left[  0,0,1\right]  ,$ $\left[  0,0,2\right]  ,$ $\left[  0,2,1\right]  $
and $\left[  0,2,4\right]  $.
\end{example}

\end{document}